\documentclass[11pt]{article}
\usepackage{amsmath, amsthm, amssymb, mathtools}
\usepackage{float,epsfig}
\usepackage{graphicx}
\usepackage{rotating}
\usepackage{subfigure}
\usepackage{caption}
\usepackage{subcaption}
\usepackage{tkz-graph}
\usepackage{color}
\usepackage{algorithm}
\usepackage{multicol}
\usepackage[utf8]{inputenc}
\usepackage{braket}
\usepackage{subcaption, multirow}
\usepackage{qcircuit}
\usepackage[overload]{empheq}
\usepackage{hyperref}
\usepackage{xcolor}
\usepackage{color,graphicx,float,lscape,enumerate}
\usepackage{algorithm}
\usepackage{svg}

\usepackage[noend]{algpseudocode}
\usepackage{mathtools}

\usepackage{hyperref}

\usepackage{xparse}

\usepackage{lipsum}



\begin{document}

\newtheorem{theorem}{\bf Theorem}[section]
\newtheorem{proposition}[theorem]{\bf Proposition}
\newtheorem{definition}[theorem]{\bf Definition}
\newtheorem{corollary}[theorem]{\bf Corollary}
\newtheorem{example}[theorem]{\bf Example}
\newtheorem{exam}[theorem]{\bf Example}
\newtheorem{remark}[theorem]{\bf Remark}
\newtheorem{lemma}[theorem]{\bf Lemma}
\newcommand{\nrm}[1]{|\!|\!| {#1} |\!|\!|}

\newcommand{\calL}{{\mathcal L}}
\newcommand{\calX}{{\mathcal X}}
\newcommand{\calY}{{\mathcal Y}}
\newcommand{\calZ}{{\mathcal Z}}
\newcommand{\calW}{{\mathcal W}}
\newcommand{\calA}{{\mathcal A}}
\newcommand{\calB}{{\mathcal B}}
\newcommand{\calC}{{\mathcal C}}
\newcommand{\calK}{{\mathcal K}}
\newcommand{\C}{{\mathbb C}}
\newcommand{\Z}{{\mathbb Z}}
\newcommand{\R}{{\mathbb R}}
\renewcommand{\SS}{{\mathbb S}}
\newcommand{\LL}{{\mathbb L}}
\newcommand{\st}{{\star}}
\def\kernel{\mathop{\rm kernel}\nolimits}
\def\sigan{\mathop{\rm span}\nolimits}

\newcommand{\klasse}{{\boldsymbol \Delta}}

\newcommand{\ba}{\begin{array}}
\newcommand{\ea}{\end{array}}
\newcommand{\von}{\vskip 1ex}
\newcommand{\vone}{\vskip 2ex}
\newcommand{\vtwo}{\vskip 4ex}
\newcommand{\dm}[1]{ {\displaystyle{#1} } }

\newcommand{\be}{\begin{equation}}
\newcommand{\ee}{\end{equation}}
\newcommand{\beano}{\begin{eqnarray*}}
\newcommand{\eeano}{\end{eqnarray*}}
\newcommand{\inp}[2]{\langle {#1} ,\,{#2} \rangle}
\def\bmatrix#1{\left[ \begin{matrix} #1 \end{matrix} \right]}
\def\cmatrix#1{\left( \begin{matrix} #1 \end{matrix} \right)}
\def \noin{\noindent}
\newcommand{\evenindex}{\Pi_e}



\def \R{{\mathbb R}}
\def \C{{\mathbb C}}
\def \F{{\mathbb F}}
\def \K{{\mathbb K}}
\def \H{{\mathbb H}}
\def \cu{\mathrm{CU}}
\def \calN{{\mathcal N}}
\def \T{{\mathbb T}}
\def \Pb{\mathrm{P}}
\def \N{{\mathbb N}}
\def \Ib{\mathrm{I}}
\def \Ls{{\Lambda}_{m-1}}
\def \Gb{\mathrm{G}}
\def \Hb{\mathrm{H}}
\def \Lam{{\Lambda}}

\def \Qb{\mathrm{Q}}
\def \Rb{\mathrm{R}}
\def \Mb{\mathrm{M}}
\def \norm{\nrm{\cdot}\equiv \nrm{\cdot}}

\def \A{{{\mathbb P}_1(\C^{n\times n})}}
\def \H{{\mathbb H}}
\def \L{{\mathbb L}}
\def \G{{\F_{\tt{H}}}}
\def \S{\mathbb{S}}
\def \s{\mathbb{s}}
\def \sigmin{\sigma_{\min}}
\def \elam{\Lambda_{\epsilon}}
\def \slam{\Lambda^{\S}_{\epsilon}}
\def \Ib{\mathrm{I}}
\def \Tb{\mathrm{T}}
\def \d{{\delta}}

\def \Lb{\mathrm{L}}
\def \N{{\mathbb N}}
\def \Ls{{\Lambda}_{m-1}}
\def \Gb{\mathrm{G}}
\def \Hb{\mathrm{H}}
\def \Delta{\triangle}
\def \Rar{\Rightarrow}
\def \p{{\mathsf{p}(\lam; v)}}

\def \D{{\mathbb D}}

\def \tr{\mathrm{Tr}}
\def \cond{\mathrm{cond}}
\def \lam{\lambda}
\def \sig{\sigma}
\def \sign{\mathrm{sign}}

\def \ep{\epsilon}
\def \diag{\mathrm{diag}}
\def \rev{\mathrm{rev}}
\def \vec{\mathrm{vec}}

\def \ham{\mathsf{Ham}}
\def \herm{\mathsf{Herm}}
\def \sym{\mathsf{sym}}
\def \odd{\mathsf{sym}}
\def \en{\mathrm{even}}
\def \rank{\mathrm{rank}}
\def \pf{{\bf Proof: }}
\def \dist{\mathrm{dist}}
\def \rar{\rightarrow}

\def \rank{\mathrm{rank}}
\def \pf{{\bf Proof: }}
\def \dist{\mathrm{dist}}
\def \Re{\mathsf{Re}}
\def \Im{\mathsf{Im}}
\def \re{\mathsf{re}}
\def \im{\mathsf{im}}

\def \sym{\mathsf{CSym}}
\def \sksym{\mathsf{skew\mbox{-}sym}}
\def \odd{\mathrm{odd}}
\def \even{\mathrm{even}}
\def \herm{\mathsf{Herm}}
\def \skherm{\mathsf{skew\mbox{-}Herm}}
\def \str{\mathrm{ Struct}}
\def \cnot{\mathrm{CNOT}}
\def \eproof{$\blacksquare$}

\def \U{\mathsf{U}}
\def \G{\mathsf{G}}
\def \bS{{\bf S}}
\def \cA{{\cal A}}
\def \E{{\mathcal E}}
\def \X{{\mathcal X}}
\def \cH{\mathcal{H}}
\def \cJ{\mathcal{J}}
\def \tr{\mathrm{Tr}}
\def \range{\mathrm{Range}}
\def \adj{\star}

\def \pal{\mathrm{palindromic}}
\def \palpen{\mathrm{palindromic~~ pencil}}
\def \palpoly{\mathrm{palindromic~~ polynomial}}
\def \odd{\mathrm{odd}}
\def \even{\mathrm{even}}
\def \QT{{\texttt{QT}}}

\newcommand{\tm}[1]{\textcolor{magenta}{ #1}}
\newcommand{\tre}[1]{\textcolor{red}{ #1}}
\newcommand{\tb}[1]{\textcolor{blue}{ #1}}
\newcommand{\tg}[1]{\textcolor{green}{ #1}}

\title{
Signed network models for dimensionality reduction of portfolio optimization}
\author{Bibhas Adhikari\thanks{E-mail:
badhikari@fujitsu.com}\\ Fujitsu Research of America Inc.\\ Santa Clara, California, USA 
}

\date{}

\maketitle
\thispagestyle{empty}

\noindent{\bf Abstract.} 
In this paper, we develop a time-series-based signed network model for dimensionality reduction in portfolio optimization, grounded in Markowitz’s portfolio theory and extended to incorporate higher-order moments of asset return distributions. Unlike traditional correlation-based approaches, we construct a complete signed graph for each trading day within a specified time window, where the sign of an edge between a pair of assets is determined by the relative behavior of their log returns with respect to their mean returns. Within this framework, we introduce a combinatorial interpretation of higher-order moments, showing that maximizing skewness and minimizing kurtosis correspond to maximizing balanced triangles and balanced $4$-cliques with specific signed edge configurations respectively. We establish that the latter leads to an NP-hard combinatorial optimization problem, while the former is naturally guaranteed by the structural properties of the signed graph model. Based on this interpretation, we propose a dimensionality reduction method using a combinatorial formulation of the mean-variance optimization problem through a combinatorial hedge score metric for assets. The proposed framework is validated through extensive backtesting on 199 S\&P 500 assets over a 16-year period (2006 - 2021), demonstrating the effectiveness of reduced asset universes for portfolio construction using both Markowitz optimization and equally weighted strategy.\\



\section{Introduction} 

Financial portfolio optimization is a fundamental problem in computational finance, with Markowitz’s mean-variance framework providing the classical formulation for balancing expected return and risk \cite{markowitz1952portfolio}. However, this framework assumes Gaussian returns and does not account for higher-order moments such as skewness and kurtosis, which are essential for capturing asymmetric and heavy-tailed behaviors observed in real financial data. As a result, extended formulations incorporating mean, variance, skewness, and kurtosis have been widely studied, leading to a multi-objective optimization problem that seeks to maximize return and skewness while minimizing variance and kurtosis, see 
\cite{jondeau2006optimal} \cite{harvey2010portfolio} \cite{ban2018machine}
\cite{brito2019portfolio} \cite{ashfaq2021gainers} \cite{zhou2021solving} \cite{uberti2023theoretical} \cite{abid2025pgp} and the references therein. Indeed, for $N$ assets in a market with $\mu\in \R^N$ denotes the mean return vector of the assets in a time window $T$, the portfolio optimization problem with no short-selling is to determine portfolio weight vector $w=[w_j,]\in\R^N_{\geq 0}$, $\sum_j w_j= 1$ such that the expected return $w^\dagger\mu$ and the skewness $S(w)=\mathbb{E}[(w^\dagger(R-\mu))^3]$  are maximized, and to minimize the variance $V(w)=w^\dagger\Sigma w$ and the kurtosis $K(w)=\mathbb{E}[(w^\dagger(R-\mu))^4].$
Here $\mathbb{E}[X]$ denotes the expectation of a random variable $X.$ Incorporating higher moments significantly increases computational and statistical complexity due to the non-convex nature of the objective and the high dimensionality of covariance and higher-order moment estimators, rendering the problem NP-hard in general \cite{murty1985some}. Consequently, dimensionality reduction and heuristic optimization methods have become essential for identifying manageable subsets of assets while preserving desirable portfolio characteristics, particularly in large financial markets.

Alternative strategies, such as equally weighted portfolios, have also demonstrated competitive performance, but they face practical limitations when managing very large asset sets \cite{demiguel2009optimal} \cite{maillard2010properties} \cite{pflug20121} \cite{gelmini2024equally}. Consequently, selecting a reduced and well-structured subset of assets for equally weighted portfolio formation incorporating mean-variance framework remains a central challenge in portfolio optimization, particularly when higher-order moment considerations are included \cite{jiang2019combining}.

The framework of combinatorial graphs or networks serves as a powerful mathematical tool across a variety of data analysis techniques. In financial applications, networks play a central role in modeling dependencies among assets via their correlation strengths. By representing assets as vertices and encoding correlations as (weighted) edges, numerous methods have been developed for tasks such as asset‐price prediction and risk analysis \cite{chi2010network} \cite{mantegna1999hierarchical} \cite{barigozzi2019nets} \cite{chen2020correlation} \cite{hautsch2015financial} \cite{peralta2016network} \cite{gregnanin2024signature}. Compared to purely statistical approaches, network analysis offers the advantage of capturing both pairwise interactions and higher‐order group dynamics among assets. Several surveys and monographs explore the role of networks in finance and economics more broadly \cite{kenett2015network} \cite{jackson2014networks}  \cite{acemoglu2015systemic}. Indeed, for finite time-series data with a time period $T<\infty,$ there is a random offset to every correlation coefficient and these values are dressed up with noise \cite{guhr2003new}, it can be validated by comparing eigenvalue density of a correlation matrix to a random matrix \cite{laloux1999noise}. An important observation from the financial data is that the effect of noise strongly depends on the ratio $N/T$, where $N$ is the size of the portfolio and $T$ the length of the available time series \cite{pafka2003noisy}, see also \cite{kondor2007noise} \cite{chung2022effects}.
There are primarily two approaches adopted to address these noise effects: noise reduction and thresholding. Noise reduction techniques aim to mitigate spurious correlations, while thresholding methods are employed to filter correlation strengths based on prescribed bounds; see \cite{kondor2007noise} \cite{millington2020partial} and the references therein. More recently, machine learning-based methods have been introduced to tackle this problem by learning robust representations of correlations directly \cite{castilho2024forecasting} \cite{mattera2023shrinkage}.

Inspired by the early work of Harary et al. on modeling financial markets using signed graphs, we adopt a weighted signed graph framework to represent financial markets. Within this framework, we demonstrate that negative edges naturally act as hedges, in alignment with Markowitz’s mean-variance portfolio theory for risk containment in portfolio formation. A primary motivation of our work is to develop a combinatorial interpretation of moment-based portfolio optimization, thereby enabling the use of graph-theoretic and combinatorial algorithms, both classical and quantum, to address such problems. This perspective provides a bridge between traditional portfolio theory and algorithmic frameworks, facilitating new computational approaches to higher-moment optimization. 

The main contribution of the paper are as follows. 
\begin{itemize}
    \item We introduce  \emph{time-evolving complete signed graphs} from financial return data, capturing local dynamics of a financial market, such as a stock market, without explicitly relying on correlation strengths. 
    \item We propose \emph{hedge score} for an asset as a measure of persistent negative co-movement structure, which serves as a first-stage filtering mechanism for dimensionality reduction of the portfolio optimization problem.
    \item We formulate a \emph{combinatorial optimization problem based on signed motifs}, providing an interpretable proxy for higher-order dependence beyond covariance.
\end{itemize}

We show that maximizing skewness and minimizing kurtosis of portfolio returns can be translated into combinatorial objectives: maximizing the number of balanced triangles and maximizing the number of balanced $4$-cliques of specific signed configuration, denoted as $K_4^{B_2}$ (see Figure \ref{fig:bK4}), in the induced signed subgraph corresponding to a risk-containing portfolio. Further, we show that the objective of finding subgraph of a given size $K$ that maximizes the density of $K_4^{B_2}$ is an NP-hard combinatorial problem. Consequently, our dimension reduction framework is currently limited to a hedge score based approach, which admits a combinatorial interpretation of risk reduction within the Markowitz mean-variance paradigm. Developing efficient classical approximation methods for counting $K_4^{B_2}$ in complete signed networks remains an open problem. In this direction, quantum-inspired techniques such as \cite{kordonowy2026perfectly} \cite{adhikari2026quantum} can offer a promising avenue for future investigation. 

To evaluate the performance of the proposed method, we analyze annual return, annual volatility, and Sharpe ratio for portfolios constructed using both Markowitz optimization and equally weighted strategies, applied to reduced asset universes as well as the full universe. We conduct extensive backtesting using 199 stocks, aligned with Google stock data, from S\&P 500 market data spanning 16 years (2006 - 2021), where portfolios are constructed in year 
$y$ and evaluated in year $y+1.$ Our numerical results indicate that portfolios formed from reduced universes consistently perform moderately better than those formed from the full universe, in all most all the years. The proposed combinatorial formulation of dimensionality reduction in portfolio optimization with higher-order moments introduces fundamentally challenging combinatorial problems, thereby opening new avenues for theoretical and algorithmic research.

Note that the proposed framework is not merely a screening method for dimension reduction, but a structural reinterpretation of portfolio selection grounded in Markowitz's theory and enriched by higher-order dependence captured through signed networks. Unlike traditional approaches that rely on pairwise dependence (e.g., clustering or sparse methods), our formulation incorporates higher-order interactions via signed motifs such as balanced triangles and $K_4^{B_2}$ patterns. The signed graph construction captures directional co-movements relative to asset-specific baselines, providing a dynamic view of hedging relationships beyond covariance. Moreover, the combinatorial formulation offers an interpretable proxy for higher moments, linking network structure to skewness and kurtosis of asset returns. Thus, the proposed approach inherits the core properties of mean-variance optimization while enriching the dependence structure through signed networks that capture higher-order interactions beyond pairwise relationships.

Furthermore, note that quantum computing constitutes a fundamentally novel paradigm for portfolio optimization. A spectrum of quantum algorithmic frameworks including quantum annealing, variational quantum algorithms such as the Quantum Approximate Optimization Algorithm (QAOA) has been employed to address Markowitz’s mean–variance problem \cite{hegade2022portfolio} \cite{leipold2024train} \cite{kordonowy2025lie} \cite{soloviev2025scaling}. These approaches are intrinsically designed for large‐scale instances, however, their implementation on fault‐tolerant hardware remains a future prospect. In contrast, Noisy Intermediate‐Scale Quantum (NISQ) platforms enable empirical assessment of both purely quantum and hybrid quantum-classical algorithms on moderately sized portfolios \cite{buonaiuto2023best}. By integrating our dimension‐reduction methodology with the operational capacity of NISQ devices, we show a promising pathway for the practical realization of hybrid quantum-classical portfolio optimization.

We emphasize that this work is an extension of our earlier study \cite{adhikari2025signed}, which was accepted and presented at the 14th International Conference on Complex Networks and their Applications (December 9–11, 2025, Binghamton, New York, USA). Compared to the conference version, the present paper introduces several significant advancements. In particular, we develop a combinatorial interpretation of higher-order moments within the proposed signed network model of financial markets (Section \ref{Sec:3}), formulate a corresponding dimensionality reduction framework based on this interpretation, and establish that the resulting $4$-clique-based combinatorial optimization problem is NP hard (Section \ref{Sec:4}). Apart for that, we validate the proposed framework on a substantially larger dataset comprising 199 assets, compared to the smaller dataset of 21 assets considered in the conference paper (Section \ref{Sec:5}). 

The remainder of the paper is organized as follows. Section \ref{Sec:2} provides a brief review of the preliminary notions and definitions used throughout the paper. Section \ref{Sec:3} introduces the proposed signed network model for financial markets, along with the notion of hedge score for assets and a combinatorial interpretation of higher-order moments, namely skewness and kurtosis. Section \ref{Sec:4} presents the proposed dimensionality reduction framework based on the corresponding combinatorial optimization formulations. Finally, Section \ref{Sec:5} provides an empirical evaluation of the proposed method using financial data from the S\&P 500 market.

\section{Preliminaries}\label{Sec:2} 

In this section, we provide a brief review of signed graphs, Markowitz’s mean–variance optimization problem, and the notation that will be used in the sequel.

A signed graph augments a standard combinatorial graph by assigning each edge a sign:  positive or negative. When vertices represent random variables, a positive (resp.\ negative) edge indicates that the corresponding variables are positively (resp.\ negatively) correlated. For a comprehensive review of signed‐graph theory and its applications, see Zaslavsky’s annotated bibliography of recent developments \cite{zaslavsky2012mathematical}. Structural balance theory, which hinges on the sign‐configuration of triangles, is fundamental in the study of signed social networks \cite{cartwright1956structural} \cite{harary1953notion}. Triangles are classified by the number of negative edges they contain: if $T_j$ denotes a triangle with $j$ negative edges for $j=0,1,2,3$, then $T_0$ and $T_2$ are balanced, whereas $T_1$ and $T_3$ are unbalanced (see Figure \ref{fig:sH_threshold} (c)). A signed graph is called \emph{balanced} if its vertex set can be partitioned into two subsets such that every positive edge lies within a subset and every negative edge connects vertices across subsets \cite{harary1953notion}. A chordal signed graph such as a complete signed graph is balanced if and only if all the triangles in it are balanced. Empirical evidence shows that real‐world signed networks are typically unbalanced, inspiring various measures to quantify this lack of balance \cite{aref2019balance} \cite{singh2017measuring} \cite{diaz2025mathematical}. 

Harary et al. \cite{harary2002signed} introduced the notion of balance signed graphs for well-structured equities portfolios that could contain risk in the portfolio. In their model, assets are considered as vertices, and the existence of positive and negative edges in the corresponding signed graph is defined by the correlation between returns of the associated pair of assets. Thus the edges indicate the tendency or manner in which the value of the assets change relative to each other. A positive edge between a pair of assets reflects that the valuation of the assets tend to move in tandem, whereas a negative edge implies that the valuations of the assets move in opposite direction, if one goes up the other goes down. Indeed, the signed graph in their model is obtained from a weighted signed graph with a thresholding function as described in Figure \ref{fig:sH_threshold} (a). Following the idea of Harary et al., a number of articles considered to investigate financial markets through signed graph models and vice versa, for instance see \cite{huffner2010separator}  \cite{aref2019balance} \cite{ehsani2020structure} \cite{figueiredo2014maximum}  \cite{vasanthi2015applications} and the references therein.    Recently, 
in \cite{bartesaghi2025global}, the authors show that the global balance index of financial correlation networks can be
used as a systemic risk measure. We note that, even though weighted correlation networks are considered in several context in the literature, weighted signed network models for financial networks are rare to find \cite{masuda2025introduction}.

Markowitz's original mean-variance model (OMV) model is formulated as 
\begin{eqnarray}
    \mbox{OMV1:} && w^* = \arg\min_{w\in\Delta} w^\dagger \widehat{\Sigma}w \,\, \mbox{s.t.} \mu^\dagger w=\ep \label{eqn:OMV1}\\
    \mbox{or OMV2:} && w^* = \arg\min_{w\in\Delta} -\mu^\dagger w + \gamma w^\dagger\widehat{\Sigma}w, \label{eqn:OMV2}
\end{eqnarray} where $\mu$ denotes the mean vector consists of the means of the asset returns and $\Delta=\left\{w\in\R^N_{\geq 0} : \sum_{i=1}^N w_i =1\right\}$ \cite{lai2022survey}. Thus Markowitz's model recommends formation of portfolio to ensure some level of $\ep$ (also called target return) of portfolio return $\mu^\dagger w$ and minimizing the portfolio variance given by equation (\ref{eqn:OMV1}), and simultaneously maximizing the return and minimizing the portfolio variance with a mixing parameter (also called \textit{risk aversion parameter}) $\gamma\in (0,\, \infty)$ in equation (\ref{eqn:OMV2}). Thus, both the models urge to gain more return and withstand less risk.

Now we describe weighted signed graph models for representing financial markets using correlation matrices.  
Since the actual correlation between the returns is unobserved, the correlation is often estimated by employing several statistical estimators \cite{mantegna1999introduction}. Denoting the unobserved covariance matrix as $\Sigma$ for a random vector $R=(R_1,\hdots,R_N),$ we denote an estimator of $\Sigma$ as $\widehat{\Sigma}=[\widehat{\Sigma}_{ij}],$ where $\widehat{\Sigma}_{ij}=\mbox{Cov}(R_i,R_j)=E[(R_i-\mu_{R_i})(R_j-\mu_{R_j})]$ denotes the estimated covariance corresponding to the random variables $R_i$ and $R_j.$ Here, $\mu_{R}=\mathbb{E}[R]$, the expected value of the random variable $R.$ In financial time-series data, let $R_i(t)$ denote the random variable corresponding to an index associated with the asset $i$ at time $t$ (for example, a day or month or year). Then a popular unbiased estimator for $\Sigma$ is the sample covariance matrix, whose entries are defined by $\widehat{\Sigma}_{ij}=\frac{1}{T-1}\sum_{t=1}^T (r_i^t-\mu_{R_i})(r_j^t-\mu_{R_j}),$ where $R_i(t)=r_i^t$ and $R_j(t)=r_j^t,$ $\mu_R=\frac{1}{T}\sum_{t=1}^T r^t,$ and $t\in\{1,\hdots, T\}$ with $T$ is the total time window. The sample correlation coefficient matrix is then defined as $\widehat{\rho}=[\widehat{\rho}_{ij}],$  with $\rho_{ij}=\mbox{Cov}(R_i,R_j)/\sqrt{\mbox{Var}(R_i)\mbox{Var}(R_j)},$ where $\mbox{Var}(X)=\frac{1}{T-1}\sum_{t=1}^T (x^t-\mu_X)^2$ is nonzero, and $\widehat{\rho}$ estimates the  population Pearson correlation matrix. Note that $-1\leq \widehat{\rho}_{ij}\leq 1$ with $\widehat{\rho}_{ij}=1$ if $i=j.$ If $\widehat{\rho}_{ij}>0$ then the random variables $X_i$ and $X_j$ are said to be positively correlated and they are negatively correlated if $\widehat{\rho}_{ij}<0.$

For financial time-series data, such as in stock market, let $S_n(t)$ denote the random variable for the price of the $n$-th stock at time $t.$ Then the random variable $R_n(t)$ which represents return of the $n$-th stock for a fixed time horizon $\Delta t$ is defined as: 
$(S_n(t+\Delta t)-S_n(t))/S_n(t)$ (Linear return) or $\log S_n(t+\Delta t) - \log S_n(t)$ (log return). Often the value of $\Delta t$ is considered as $1$.  For Markowitz's portfolio theory applications, a correlation coefficient estimator matrix must be non-singular, and hence positive definite. We mention here that there are other powerful methods to model the return time-series, such as the GARCH process introduced by Bollerslev \cite{bollerslev1986generalized}, a generalization of the ARCH process proposed by Engle in \cite{engle1982autoregressive}.

\section{Signed network models for local dynamics of financial markets}\label{Sec:3}

In this section, we introduce exploring a time-series of signed networks to capture the local dynamics of asset returns in a financial market, which is the building block of our proposal of dimensionality reduction of portfolio optimization. First, we establish that negative edges in the standard weighted signed graph $G^s({\widehat{\Sigma}_D})$ acts as hedges in a portfolio. Here a weighted signed graph $G^s({\widehat{\Sigma}_D})$ represents a model financial market associated with a (denoised) correlation estimator matrix $\widehat{\Sigma}_D=[\widehat{\Sigma}_{ij}^{D}]$ as follows. 


\begin{definition}\label{def:wsg}(Weighted signed graph models of financial markets)
The vertex set of $G^s({\widehat{\Sigma}_D})$ is the set of assets in a portfolio index by $1,2,\hdots, N.$ Then the edge set $E\subseteq V\times V$ is defined by the two following ways.    
\begin{enumerate}
    \item Without thresholding: there is an edge between a pair of vertices $(i,j)$ if and only if $\widehat{\Sigma}_{ij}^D\neq 0.$ The sign of an edge $(i,j)$ is positive if $\widehat{\Sigma}_{ij}^D> 0$ and negative if $\widehat{\Sigma}_{ij}^D< 0.$ The weight of the edge is $\widehat{\Sigma}_{ij}^D.$ 
    \item With thresholding: let $0<\tau_+<1$ and $-1<\tau_-<0.$ Then there is a positive edge for the vertex pair $(i,j)$ with weight $\widehat{\Sigma}_{ij}^D$ if $\widehat{\Sigma}_{ij}^D>\tau_+$ and a negative edge for the vertex pair $(i,j)$ with weight $\widehat{\Sigma}_{ij}^D$ if $\widehat{\Sigma}_{ij}^D<\tau_-.$
\end{enumerate}
\end{definition}

A signed graph representation of a financial market is the underlying signed graph obtained by relaxing the edge weights of a weighted signed portfolio graph. This can be achieved in two ways: directly from the estimated correlation matrix with thresholding and from the denoised correlation matrix. In both cases, the threshold function may or may not be applied.  In \cite{harary2002signed}, Harary et al. considered using a threshold function directly from the estimated correlation matrix as described in Figure \ref{fig:sH_threshold} (a). As they explained, the edges in the normalized market graph represent the tendency of the return values of the associated assets (vertices). 

We now establish, from a risk-containment perspective, that negative edges in a signed graph representation play a crucial role in reducing portfolio risk compared to portfolios composed solely of positively correlated assets. We adopt portfolio variance as the measure of risk, following Markowitz’s Portfolio Theory (MPT) \cite{markowitz1952portfolio}. 
According to MPT, the objective of a diversified investor is to minimize portfolio variance. The minimum-variance portfolio problem can be formulated as
\[
\min_{w} \; w^\dagger \widehat{\Sigma} \, w
\quad \text{subject to} \quad
\mathbf{1}^\dagger w = 1,
\]
where $\widehat{\Sigma}$ denotes the estimated covariance matrix of asset returns, $w = [w_1, \ldots, w_N]^T$ is the portfolio weight vector with $w_j \ge 0$ representing the fraction of capital invested in asset $j$, and $\mathbf{1}$ is the all-one vector of dimension $N$, the total number of assets. The constraint $w_j \ge 0$ enforces the no-short-selling condition. We now state the following theorem, whose proof is provided in \cite{adhikari2025signed}.

\begin{theorem}\label{thm:var}
Let $w=[w_1,\hdots,w_N]^\dagger$ with $w_i\geq 0$ and $\sum_{i=1}^N w_i=1.$ Suppose $G^s(\widehat{\Sigma})$ is the underlying (weighted) signed graph with at least one negative edge. Then $w^\dagger \widehat{\Sigma}w\leq w^\dagger |\widehat{\Sigma}|w,$ where  $|\widehat{\Sigma}|=[|\widehat{\Sigma}_{ij}|].$
\end{theorem}

Theorem \ref{thm:var} affirms that negative edges act like hedges in a portfolio, as defined in \cite{baur2010gold}. Now note that the sample covariance of return values of a pair of assets is given by $\widehat{\Sigma}_{ij}=\frac{1}{T-1}\sum_{t=1}^T (R_i^t-\mu_{R_i})(R_j^t-\mu_{R_j})$ for a time period $T$, where $R_k^t$ denotes the return of asset $k$ at time $t,$ and $\mu_{R_k}$ is the mean of the return values of the asset $k$ for the time period $T.$ If $\widehat{\Sigma}_{ij}<0,$ it indicates that one of the assets had a few `bad days' compare to its own mean return value than the other asset in terms of their  return values, although for the other days their return values could be at per compare to their own mean return values. Whereas, if $\widehat{\Sigma}_{ij}>0$ then it would mean that they have the same `bad days' and `good days' i.e. return values of both the assets go up or down together corresponding to their own mean return values in most of the days or the values go up or down quite deep together on a few days compare to the days when pairwise go in opposite directions making a pair (up,down) or (down,up). In an extreme case, one ``very good" or ``very bad" day of either or both the assets can flip the sign of $\widehat{\Sigma}_{ij}$ from positive to negative or vice-versa. By compressing these finer co‐movement patterns into $\widehat{\Sigma}$, the Markowitz mean-variance formulation masks this local return dynamics. This interpretation applies equally to raw and denoised (or thresholded) covariance estimators; henceforth, ``covariance matrix" refers to either form.

To capture the local dynamics of asset returns in a financial market, we model the market as a time-indexed sequence of complete signed graphs, defined as follows.

\begin{definition}\label{def:sg}(Time-series of signed graphs)
The signed graph $G^s_t(\boldsymbol{\mu,R_N})=(V, E_t)$ of $N$ assets at a time $t\in\{1,\hdots,T\}$ with $V=\{1,\hdots,N\}$ as the set of assets, $\mu$ is the mean return vector of the assets and $\boldsymbol{R_N}=(R_1^t,\hdots,R_N^t)$ is the observed empirical return values. For a pair of assets $(i,j)$ there is a positive edge if $(R_i^t-\mu_{R_i})(R_j^t-\mu_{R_j})\geq 0$ and a negative edge if $(R_i^t-\mu_{R_i})(R_j^t-\mu_{R_j})< 0.$    
\end{definition}

Figure~\ref{fig:sH_threshold} (b) presents the symbolic representation of the positive and negative edge assignments in our proposed signed graph model, while Figure~\ref{fig:sH_threshold} (a) depicts the corresponding edge assignment framework of Harary \emph{et al.} Then we have the following theorem.

\begin{theorem}
 The signed graph $G^s_t(\boldsymbol{\mu,R_N})$ is balanced for every $t.$   
\end{theorem}

\pf Since $G^s_t(\boldsymbol{\mu,R_N})$ is a complete graph, it is chordal, and hence it is balanced if every triangle subgraph is balanced. We prove the statement by method of contradiction. If possible let there be a $T_1$ type triangle $(i,j,k)$ formed by the vertices $i,j,k$ such that the edge $(i,k)$ is negative, and the edges $(i,j)$ and $(j,k)$ are positive. Consequently, $(R_i^t-\mu_i)(R_k^t-\mu_k)<0$, $(R_i^t-\mu_i)(R_j^t-\mu_j)>0$, and $(R_j^t-\mu_j)(R_k^t-\mu_k)<0.$ Now if $(R_i^t-\mu_i)>0$ then $(R_k^t-\mu_k)<0$ and $(R_j^t-\mu_j)>0$ follows from the first two signs, which contradicts the fact that $(R_j^t-\mu_i)(R_k^t-\mu_k)<0.$ Similar arguments also hold true if $(R_i^t-\mu_i)<0.$ Thus a triangle of type $T_1$ can not exist in $G^s_t(\boldsymbol{\mu,R_N})$. Next, suppose there is a $T_3$ type triangle on three vertices $i,j,k$ in $G^s_t(\boldsymbol{\mu,R_N})$  such that the edges $(i,k),$ $(i,j)$ and $(j,k)$ are negative. This implies, $(R_i^t-\mu_i)(R_k^t-\mu_k)<0$, $(R_i^t-\mu_i)(R_j^t-\mu_j)<0$, and $(R_j^t-\mu_j)(R_k^t-\mu_k)<0.$ Now, if $(R_i^t-\mu_i)>0$ then $(R_j^t-\mu_j)>0$ and $(R_k^t-\mu_k)>0$ which contradicts $(R_j^t-\mu_j)(R_k^t-\mu_k)<0.$ A similar conclusion is also true if $(R_i^t-\mu_i)<0.$ This concludes the proof. \hfill{$\square$}

\begin{figure}[htbp]
    \centering
    \subfigure[]{\includegraphics[width=0.25\textwidth]{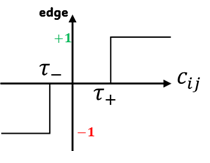}}
    ~ ~ ~
    \subfigure[]{\includegraphics[width=0.20\textwidth]{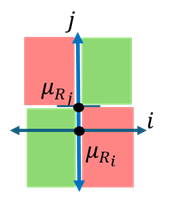}} 
     ~ ~ ~
    \subfigure[]{\includegraphics[width=0.40\textwidth]{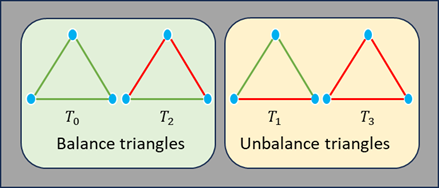}} 
    \caption{(a) Threshold function \cite{harary2002signed} for signed network formation. $c_{ij}$ denotes the covariance or correlation strength for the assets $i$ and $j$, (b) The formation of positive and negative edges in $G^s_t(\boldsymbol{\mu,R_N})$, $(c)$ Signed triangles in  signed graphs, green and red colored edges represent the positive and negative edges respectively.}
    \label{fig:sH_threshold}
\end{figure}

\subsection{Hedge score of assets}

In the weighted graph representation of a portfolio, we observe that a negative edge helps to reduce portfolio risk. As proved in Theorem \ref{thm:var}, for any invest allocation vector, the risk can be contained more by having negative edges (negatively correlated assets) than positively correlated edges (positively correlated assets) of equal strengths. Observing this, we define \textit{hedge score} of an asset in a financial market based on our formulation of $G^s_t(\boldsymbol{\mu,R_N}).$

\begin{definition}\label{def:hs}(Hedge score)
Let $S^t_{n}:V\setminus\{n\}\rightarrow \{0,1\}$ be a function $S^t_n(j)=1$ if $(R_n^t-\mu_{R_n})(R_j^t-\mu_{R_j})<0$ and $S^t_n(j)=0$ otherwise, where $t\in\{1,\hdots,T\}$. Then the hedge score of an asset $n$ is defined as
\begin{equation}\label{eqn:hs}
h(n,T)= \frac{\sum_{\substack{j\in V\\ j\neq n}} \sum_{t=1}^{T} S_n^t(j)}{T(N-1)}. 
\end{equation}
\end{definition}

Note that $S_n^t$ counts the negative degree of the vertex $n$ in the graph $G^s_t(\mu,R_N).$ Besides, $0\leq h(n,T)\leq 1.$

\subsection{Higher moments}

In this section, we incorporate higher-order moments of asset return distributions into our signed-network framework and develop a corresponding combinatorial optimization formulation. In standard higher-order portfolio optimization models, the third and fourth central moments, namely skewness and kurtosis, are included in the objective function. Recall that skewness and kurtosis of a portfolio with weight vector $w$ are defined as
\[
S(w) = \mathbb{E}\!\left[(w^\dagger (R-\mu))^3\right]
\quad \text{and} \quad
K(w) = \mathbb{E}\!\left[(w^\dagger (R-\mu))^4\right],
\]
respectively. Expanding these expressions yields
\begin{eqnarray}
S(w) &=& \sum_{i=1}^N \sum_{j=1}^N \sum_{k=1}^N 
w_i w_j w_k \, \mathbb{E}\!\left[(R_i-\mu_i)(R_j-\mu_j)(R_k-\mu_k)\right], \label{eqn:skew}\\
K(w) &=& \sum_{i=1}^N \sum_{j=1}^N \sum_{k=1}^N \sum_{l=1}^N 
w_i w_j w_k w_l \, \mathbb{E}\!\left[(R_i-\mu_i)(R_j-\mu_j)(R_k-\mu_k)(R_l-\mu_l)\right].
\label{eqn:kurt}
\end{eqnarray}

Now to explore capturing the formulation of portfolio optimization including higher moments through our network formulation that aims to selection of assets  that maximize $S(\boldsymbol{w})$ and minimize $S(\boldsymbol{w}).$ Assuming $w_i\geq 0$ for any asset in the portfolio i.e. without short-selling, we have the following observation. 

We first consider the skewness expression in Eq.~(\ref{eqn:skew}). The objective is to maximize the expected value of 
\[
S_{ijk} := (R_i-\mu_i)(R_j-\mu_j)(R_k-\mu_k)
\]
for any triple of assets $(i,j,k)$ over a time window $T$, where the local time index $t \in \{1,2,\ldots,T\}$ represents a day, month, or year, depending on the sampling resolution. Note that in the algebraic expansion of $S(\boldsymbol{w})$, the indices $i,j,k$ are not required to be distinct. In contrast, in the signed network $G^s_t(\boldsymbol{\mu},R_N)$, we seek a combinatorial interpretation of the quantity $S_{ijk}^t$ at each time step $t$. Recall that $R_a^t$ denotes the return value of asset $a$ at time $t$ within the window $T$. The following cases arise.

\begin{itemize}
\item \textbf{Case I:} Two indices are equal. Without loss of generality, let $i=j$. Then
\[
S_{ijk}^t = (R_i^t-\mu_i)^2 (R_k^t-\mu_k),
\]
whose sign depends entirely on the deviation of the $k$-th asset’s return from its mean over the time window.

\item \textbf{Case II:} All three indices are equal, i.e., $i=j=k$. Then
\[
S_{ijk}^t = (R_i^t-\mu_i)^3,
\]
whose sign again depends solely on the deviation of that asset’s return from its mean.

\item \textbf{Case III:} All three indices are distinct. Then the sign of $S_{ijk}^t$ is positive precisely when the triangle induced by vertices $i,j,k$ in the signed graph $G^s_t(\boldsymbol{\mu},R_N)$ is balanced, i.e., when it is of type $T_0$ or $T_2$.
\end{itemize}

We next consider the kurtosis expression in Eq.~(\ref{eqn:kurt}). Let
\[
K_{ijkl} := (R_i-\mu_i)(R_j-\mu_j)(R_k-\mu_k)(R_l-\mu_l),
\]
and seek a network-based interpretation for minimizing the expected value of $K(\boldsymbol{w})$. As before, the indices $(i,j,k,l)$ need not be distinct. The following cases determine the sign behavior relevant for kurtosis reduction.

\begin{itemize}
\item \textbf{Case I:} Two indices are equal and the other two are distinct; without loss of generality, let $i=j$ and $k \neq l$. Then
\[
K_{ijkl}^t = (R_i^t-\mu_i)^2 (R_k^t-\mu_k)(R_l^t-\mu_l),
\]
which becomes negative and hence reduces kurtosis when the edge between assets $k$ and $l$ carries a negative sign in the signed graph.

\item \textbf{Case II:} All four indices are distinct. In this case, $K_{ijkl}^t$ is negative when the subgraph induced by vertices $i,j,k,l$ i.e., the signed $4$-clique is isomorphic to $K_4^{B_4}$ as described in Figure \ref{fig:bK4}. Indeed, note that $K_{ijkl}^t<0$ if odd number i.e. either $1$ or $3$ of its terms are $<0.$ If only one term is negative, without loss of generality, suppose $(R_i^t-\mu_i)<0$ and all other terms are positive. Then $(R_i^t-\mu_i)(R_j^t-\mu_j)<0$, $(R_i^t-\mu_i)(R_k^t-\mu_k)<0,$ $(R_i^t-\mu_i)(R_l^t-\mu_l)<0,$ $(R_j^t-\mu_k)(R_k^t-\mu_k)>0,$ $(R_j^t-\mu_j)(R_l^t-\mu_l)>0,$ $(R_k^t-\mu_k)(R_l^t-\mu_l)>0,$ which implies a $4$-clique on vertices $i,j,k,l$ is isomorphic to $K_4^{B_2}.$ Next, if $3$ terms are negative, without loss of generality, assume that $(R_l^t-\mu_l)>0$ and all other terms are negative. Then $(R_i^t-\mu_i)(R_j^t-\mu_j)>0$, $(R_i^t-\mu_i)(R_k^t-\mu_k)>0,$ $(R_i^t-\mu_i)(R_l^t-\mu_l)<0,$ $(R_j^t-\mu_j)(R_k^t-\mu_k)>0,$ $(R_j^t-\mu_j)(R_l^t-\mu_l)<0,$ $(R_k^t-\mu_k)(R_l^t-\mu_l)<0,$ and the vertices form a $4$-clique isomorphic to $K_4^{B_2}.$   
\end{itemize}

These observations indicate that a desirable portfolio should favor asset subsets that include balanced triangles in the time series of signed graphs $G^s_t(\boldsymbol{\mu},R_N)$, thereby promoting higher skewness, which is obviously true since $G^s_t(\boldsymbol{\mu},R_N)$ is a balanced complete graph. On the other hand, kurtosis reduction is associated with the presence of negative edges and a large number of balance $4$-cliques of type $K_4^{B_2}.$ Note that $4$-cliques of type $K_4^{B_j},$ $j=1,2,3,4$ can exist in $G^s_t(\boldsymbol{\mu},R_N)$ (see Figure \ref{fig:bK4}) without compromising on higher skewness, however maximizing $K_4^{B_2}$ type $4$-cliques indicate lower value of kurtosis. Since the presence of negative edges is incorporated withing the definition of hedge-score, and   the other combinatorial characterizations could be leveraged to formulate an optimization framework for dimensionality reduction in portfolio optimization incorporating higher-order moments.

\begin{figure}[htbp]
    \centering
    {\includegraphics[width=0.40\textwidth]{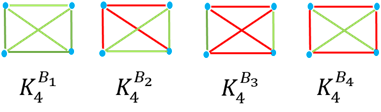}}
    \caption{The possible $4$-cliques in $G^s_t(\boldsymbol{\mu}, R_N)$. Edges in green and red indicate positive and negative edges respectively.}
    \label{fig:bK4}
\end{figure}

\section{Dimensionality reduction of portfolio optimization}\label{Sec:4}

In this section, we develop a dimensionality reduction framework based on the proposed hedge-score formulation derived from Markowitz’s portfolio optimization model, and further extend it by incorporating higher-order moments into the portfolio selection process.

The dimensionality reduction formulation for Markowitz’s optimal mean-variance (OMV) problem is based on the time series of signed graphs $G^s_t(\boldsymbol{\mu}, R_N)$ over a time window $T$, as introduced in equation (\ref{eqn:opts}). The formulation relies on hedge scores of assets, and the objective function is defined by maximizing the expected return weighted by hedge scores, which capture risk-minimization behavior. In our signed network representation, risk containment is reflected through the presence of negative edges in $G^s_t(\boldsymbol{\mu}, R_N)$.

\subsection{Dimensionality reduction based on Markowitz’s OMV model}

Motivated by the OMV problem stated in equation (\ref{eqn:OMV2}), we propose the following optimization problem, which effectively reduces the dimensionality of Markowitz’s portfolio optimization task. Solving this problem yields a subset of assets that simultaneously promotes higher expected return and lower risk under the proposed signed-network formulation.

\begin{equation}\label{eqn:opts}
\text{OPT1:} \quad 
\arg\max_{S \subseteq V} \sum_{n \in S} h(n,T)\,\mu_{R_n}(T)
\;=\;
\arg\max_{S \subseteq V} H_S^\dagger(T)\,\mu_S(T),
\end{equation}
where $\mu_{R_n}(T)$ denotes the mean return of asset $n$ over the time window $T$, and $N = |V|$ is the total number of assets in the market. For a subset of assets $S \subseteq V$, we define
\[
H_S(T) = [\,h(s_1,T), \ldots, h(s_{|S|},T)\,]^\dagger
\quad \text{and} \quad
\mu_S(T) = [\,\mu_{R_{s_1}}(T), \ldots, \mu_{R_{s_{|S|}}}(T)\,]^\dagger
\]
as the column vectors of hedge scores and mean returns, respectively, for assets $s_k \in S$ over the period $T$. The solution set of Eq.~(\ref{eqn:opts}) may be interpreted as a potential hedge-protected diversified asset universe.

The theoretical maximum of Eq.~(\ref{eqn:opts}) would occur when the signed graphs $G^s_t(\boldsymbol{\mu}, R_N)$ are complete with all edges negative for all $t$, although such a configuration is not realizable in practice for moderately large asset sets. Therefore, we impose a cardinality constraint $|S| = K \ll N$ in Eq.~(\ref{eqn:opts}) to obtain a reduced asset universe suitable for subsequent portfolio optimization.

Once the asset subset $S$ is determined by solving OPT1, the investment weight vector $w$ can be computed using standard allocation strategies, such as equal weighting ($1/|S|$) or Markowitz mean–variance optimization restricted to the reduced universe. Algorithm~\ref{Alg:dr_omv} summarizes the procedure for selecting the top $K$ assets that solve OPT1. The computational complexity per time step is $O(N)$, resulting in a total complexity of $O(TN)$ over the window $T$.

\begin{algorithm}
  \caption{Algorithm for OMV dimensionality reduction}
  {\bf Input:} Data of daily returns
  \begin{algorithmic}[1]
    \State Extract daily log returns $R_i^t$ for the time-period $T$
    \State Compute mean returns $\mu_i$
    \State Initialize $d_i^-=0$ for all assets
    \State For each local-time (for example, day) $t:$
    \begin{itemize}
        \item[-] Compute the deviations: $\delta_i^t=R_i^t-\mu_i$
        \item[-] For each asset $i:$ $d_i^- += \#\{j\neq i: \delta_i^t\delta_j^t <0\}$
    \end{itemize}
    \State Compute hedge-scores $h_i=d_i^-/(T(N-1))$
    \State compute $s_i=h_i\mu_i$
    \State Return top $K$ assets by $s_i$
  \end{algorithmic}
  \label{Alg:dr_omv}
\end{algorithm}

\subsubsection{Practical implementation of Algorithm \ref{Alg:dr_omv}}\label{Sec:K}
Now we discuss how to choose the parameters for OMV dimension reduction.

The parameter $T$ denotes the length of the training window used to construct the signed network representation of the market. In our empirical implementation in Section \ref{Sec:5}, $T$ corresponds to the number of trading days within a fixed calendar year, ensuring that the estimation of asset-specific statistics.  Larger values of $T$ provide more stable estimates of mean returns and co-movement patterns, while smaller windows allow the model to adapt more quickly to regime shifts. In practice, the use of a rolling yearly window offers a tractable compromise and aligns with standard conventions in empirical asset pricing.

For each asset $i$, we compute the hedge score $h_i$ as a measure of persistent negative co-movement with the rest of the market. Specifically, let $R_i^t$ denote the daily log return of asset $i$ on day $t$, and let $\mu_i$ denote its mean return over the training window. For each day $t$, we construct a complete signed graph by assigning a positive edge between assets $i$ and $j$ if $(R_i^t - \mu_i)(R_j^t - \mu_j) \geq 0$, and a negative edge otherwise. The negative degree of asset $i$ on day $t$ is defined as the number of assets with which it has a negative edge. The hedge score is then given by
\[
h_i = \frac{1}{T(N-1)} \sum_{t=1}^T d_i^{-}(t),
\]
where $d_i^{-}(t)$ denotes the negative degree of asset $i$ on day $t$. This normalization ensures comparability across assets and time periods. Intuitively, $h_i$ captures the extent to which an asset systematically moves in opposition to others, and thus serves as a proxy for its hedging potential.

Next, we consider choices for the value of $K$ for dimension reduction. There could be multiple possible ways to set the value of $K.$ To complement heuristic choices, the parameter $K$ can be selected using quantitative rules based on both score distributions and network structure. Let $\{s_{(1)} \geq s_{(2)} \geq \cdots \geq s_{(N)}\}$ denote the ordered hedge-adjusted scores $s_i = h_i \mu_i$. A score-based rule selects
\[
K = \max \left\{ k : \Delta s_{(k)} = s_{(k)} - s_{(k+1)} \geq \tau \right\},
\]
for a threshold $\tau > 0$, capturing the ``elbow'' point beyond which marginal gains diminish. Alternatively, a percentile rule sets
\[
K = \left\lceil \alpha N \right\rceil, \quad \alpha \in (0,1),
\]
retaining a fixed fraction of top-ranked assets.These rules provide systematic and interpretable mechanisms for determining $K$, balancing statistical significance and portfolio performance. Among all the choices, an investor's input for an upper bound on the size of a portfolio may also be incorporated. In our numerical simulation, we show how different heuristic choices of $K$ impact the performance of the portfolio in out-of-sample data in Section \ref{Sec:5}.

\subsection{Dimensionality reduction based on higher moments}

We next propose a dimensionality reduction framework for higher-moment portfolio optimization based on the time series of signed graphs $G^s_t(\boldsymbol{\mu}, R_N)$. 

Motivated by the observations in Section~\ref{Sec:3}, we formulate the following optimization problem. Since our primary objective is dimensionality reduction, we seek a subset of assets of size $\leq K \ll N$ that preserves key combinatorial structures, namely, a large number of balanced triangles and balanced $4$-cliques of type $K_4^{B_2}$ within the induced signed subgraphs. We therefore propose the following objective:
\begin{equation}\label{eqn:opts2}
\text{OPT2:} \quad
\arg\max_{\substack{S \subseteq V \\ |S|\leq K}}
\left[\; H_S^\dagger(T)\,\mu_S(T)
+  \sum_{t=1}^T \dfrac{\left|\mathcal{K}_4^{t,S,B_2}\right|}{{|S|\choose 4}}\right],
\end{equation}
where $\mathcal{K}_4^{t,S,B_2}$ is the set of all $4$-cliques isomorphic to $K_4^{B_2}$ induced by vertices in $S$ for the local time $t$ and ${|S|\choose 4}$ is the total number of cliques in a complete graph on $|S|$ vertices, respectively. The choice of $K$ can be determined either through data-driven procedures or based on investor preferences, as discussed in Section~\ref{Sec:K}.

This formulation introduces, to the best of our knowledge, new combinatorial optimization problem centered on maximizing density of $4$-cliques of specific signed configuration in signed graphs, particularly complete signed graphs. The closest related problems in the unsigned setting include the $k$-clique-densest-subgraph problem, which attracted significant attention in graph mining due to their wide range of applications \cite{tsourakakis2015k}. In contrast, motif-density problems in signed graphs remain largely unexplored.

Now we discuss the hardness of finding a subset $S$ of vertices with $|S|=K$ of a complete signed graph on $N$ vertices that maximizes the density of $4$-cliques of type $K_4^{B_2}$. First we formalize the signed edge configuration of the $4$-clique $K_4^{B_2}.$

\begin{definition}
Let $G^s=(V,\sigma)$ be a complete signed graph with sign function 
$\sigma : V \times V \to \{+1,-1\}$. 
A set of four vertices $\{i,j,k,l\} \subseteq V$ induces a balanced signed clique of type $K_4^{B_2}$ if there exists a vertex, say $i$, such that $\sigma(i,j)=\sigma(i,k)=\sigma(i,l)=-1$,
and $\sigma(j,k)=\sigma(j,l)=\sigma(k,l)=+1.$
Equivalently, a $K_4^{B_2}$ consists of a positive triangle together with a fourth vertex that is negatively connected to all vertices of that triangle.
\end{definition}

Now we recall the CLIQUE problem \cite{karp2009reducibility}.

\begin{definition}[CLIQUE problem]
Given an undirected graph $H=(V_H,E_H)$ and an integer $k$, the CLIQUE problem asks whether there exists a subset $C \subseteq V_H$ such that $|C| = k$ and
$(u,v) \in E_H \quad \forall u,v \in C, \ u \neq v.$ Such a subset is called a clique of size $k$.
\end{definition}

In the following theorem we show that finding a subset $S$ of size $K<N$ that maximizes $K_4^{B_2}$ in a complete signed graph on $N$ vertices is NP hard.

\begin{theorem}
Let $G^s=(V,\sigma)$ be a complete signed graph on $N$ vertices. 
The optimization problem of selecting a subset $S \subseteq V$ of size $K$ that maximizes the number of induced $K_4^{B_2}$ patterns is NP-hard.
\end{theorem}

\begin{proof}
We prove NP-hardness by reduction from the CLIQUE problem.

Let $H=(V_H,E_H)$ be an instance of CLIQUE with $|V_H|=n$. Construct a complete signed graph $G=(V,\sigma)$ as follows: $V = V_H \cup \{l\}$, where $l$ is a new vertex. Then define the sign function $\sigma$ as follows: $\sigma(u,v)=+1$ if $(u,v) \in E_H,$ and $\sigma(u,v)=-1$ otherwise, for all $u,v \in V.$

Now set the subset size $K = c+1$ for some positive integer $c<n$ and set $T = \binom{c}{3}.$
First, suppose $H$ contains a clique $C \subseteq V_H$ with $|C|=c<n$. 
Consider the subset $S = C \cup \{l\}.$

Then every triple $\{i,j,k\} \subseteq C$ forms a positive triangle, and since $l$ is negatively connected to all vertices of $C$, each set $\{i,j,k,l\}$ forms a $K_4^{B_2}$. Thus $S$ contains exactly $T$ copies of $K_4^{B_2}$.
Conversely, suppose there exists a subset $S \subseteq V$ of size $K=c+1$ containing at least $T$ copies of $K_4^{B_2}$. 
Each $K_4^{B_2}$ consists of a positive triangle and a vertex negatively connected to that triangle. Since vertex $l$ is negatively connected to all vertices of $V_H$, achieving $T=\binom{c}{3}$ such patterns requires a subset of $c$ vertices in $V_H$ that induces $\binom{c}{3}$ positive triangles. This is possible only if all $\binom{c}{2}$ edges among those vertices are positive, implying that they form a clique of size $c$ in $H$.

Thus, solving the $K_4^{B_2}$ maximization problem would solve CLIQUE. Since CLIQUE is NP-complete, the maximization problem is NP-hard.
\end{proof}

However, in the following lemma, we show that counting $K_4^{B_2}$ in a given subset of vertices could be done in polynomial time.

\begin{lemma}
Let $G^s=(V,\sigma)$ be a complete signed graph and let $S \subseteq V$ with $|S|=K$. 
The number of induced $K_4^{B_2}$ patterns in $S$ can be computed in $O(K^4)$ time.
\end{lemma}

\begin{proof}
There are exactly $\binom{K}{4}$ subsets of size four in $S$. For each such subset, checking whether it forms a $K_4^{B_2}$ requires examining the six edge signs, which takes constant time. Therefore the total running time is $O(K^4)$.
\end{proof}

\subsubsection{Optimization workflow for dimension reduction}

Although OPT2 is NP-hard, and the development of approximation algorithms with provable guarantees remains an open problem, we propose a structured optimization workflow aimed at dimension reduction, within which OPT1 naturally arises as a preliminary step. 

Given a complete signed graph $G=(V,\sigma)$, the problem OPT2 seeks to identify a subset $S \subseteq V$ whose cardinality is bounded above by $K$ such that the induced subgraph on $S$ jointly maximizes the number of balanced signed $4$-cliques of type $K_4^{B_2}$ and the cumulative hedge scores of the selected assets. Formally, this corresponds to a higher-order combinatorial optimization problem that captures both hedging characteristics and higher-order dependence structure encoded by signed motifs. The NP-hardness of OPT2 implies that exact optimization is computationally infeasible for large-scale instances. Consequently, an important direction for future research is the development of principled approximation methods for estimating OPT2. Potential approaches include continuous relaxations and semidefinite programming formulations, spectral and graph-theoretic heuristics that exploit the structure of signed networks, and submodularity-inspired approximations based on suitable surrogate objective functions. Establishing approximation guarantees, as well as identifying structural regimes under which efficient algorithms perform well, remains a promising and largely unexplored avenue for advancing network motif based portfolio selection.

To address the computational challenges, one could adopt a two-stage framework that serves as a practical approximation scheme. In the first stage (OPT1), perform a coarse filtering of the asset universe by selecting the top $K_1$ assets according to the score $s_i = h_i \mu_i$, which combines hedging potential and average return. This step reduces the dimensionality of the problem while preserving economically relevant assets. In the second stage, the selection can be refined by solving a restricted combinatorial optimization problem over subsets of size $K_2 \leq K_1$. The objective is to maximize the number of $K_4^{B_2}$ patterns aggregated over all daily signed graphs in the training period, thereby capturing higher-order co-movement structures within a tractable candidate universe. Due to the classical hardness of this problem, alternative methods such as quantum-classical methods developed in \cite{kordonowy2026perfectly} \cite{adhikari2026quantum} can be extended to estimate $K_4^{B_2}$ counts in complete signed networks.

\section{Empirical analysis}\label{Sec:5}

Since OPT2 in Eq.~(\ref{eqn:opts2}) is computationally hard, in this paper, we consider finding a reduced asset universe employing the OMV reduction Algorithm  \ref{Alg:dr_omv}. This hedge 
score-based optimization  provides a  reduced universe $S_K$ with a desired number of top $K$ assets selected by hedge scaled expected returns over the entire period of time $T.$ The complete subgraph induced by $S$ is balance due to its construction which yields a positive value of skewness obtained by the combinatorial analogue of balance triangles in the reduced universe. Besides, the reduced universe implicitly guarantees existence of negative edges, which is one of the conbinatorial interpretation of minimizing kurtosis as explained in the previous section.

\subsection{Data, methodologies and backtesting results}
To test the proposed methodology for dimensionality reduction and portfolio construction, we consider the dataset of  assets whose data are aligned with the Google stock in S\&P500 index from August 2006 to Dec 2021. This data forms a universe of 199 assets obtained from Yahoo finance and publicly available in \href{https://www.kaggle.com/datasets/paultimothymooney/stock-market-data/data}{Kaggle}. The assets are given by 'A', 'AAP', 'ABMD', 'ABT', 'ACN', 'ADI', 'ADM', 'ADP', 'ADSK', 'AJG', 'AKAM', 'ALB', 'ALGN', 'ALK', 'AMAT', 'AMD', 'AME', 'AMGN', 'AMT', 'AMZN', 'AOS', 'APA', 'APD', 'ARE', 'ATVI', 'AVY', 'BAC', 'BAX', 'BBY', 'BDX', 'BEN', 'BIIB', 'BIO', 'BRK-A', 'BSX', 'BWA', 'BXP', 'CAG', 'CB', 'CCI', 'CDE', 'CHD', 'CHRW', 'CINF', 'CLX', 'CMI', 'CNC', 'COO', 'COP', 'CPB', 'CPRT', 'CRM', 'CSCO', 'CTAS', 'CTSH', 'CUK', 'D', 'DGX', 'DOV', 'DPZ', 'DVA', 'EA', 'EBAY', 'ECL', 'EFX', 'EL', 'EMN', 'ES', 'EW', 'EXR', 'FAST', 'FIS', 'FISV', 'FITB', 'FLS', 'FMC', 'FTI', 'GGG', 'GILD', 'GIS', 'GOOG', 'GPC', 'GPN', 'GWW', 'HAS', 'HBAN', 'HD', 'HES', 'HRB', 'HRL', 'HST', 'HSY', 'HUM', 'IDXX', 'IFF', 'ILMN', 'ISRG', 'ITW', 'IVZ', 'JBHT', 'JCI', 'JKHY', 'JNPR', 'JPM', 'K', 'KIM', 'KMB', 'KSS', 'LEG', 'LH', 'LNC', 'LNT', 'LOW', 'MAA', 'MAR', 'MCHP', 'MCO', 'MDLZ', 'MLM', 'MMC', 'MOS', 'MSFT', 'NEE', 'NEOG', 'NFLX', 'NI', 'NOC', 'NOV', 'NTAP', 'NTRS', 'NVR', 'NWL', 'O', 'ODFL', 'OMC', 'ORLY', 'OXY', 'PAYX', 'PCAR', 'PH', 'PHM', 'PKG', 'PKI', 'PLD', 'PNW', 'PPG', 'PRU', 'PVH', 'RCL', 'REG', 'RF', 'RHI', 'RLI', 'ROK', 'ROL', 'ROP', 'SBUX', 'SCHW', 'SEE', 'SHW', 'SIVB', 'SLB', 'SLG', 'SNPS', 'SO', 'SPG', 'SRE', 'STT', 'SWK', 'SYK', 'T', 'TJX', 'TMO', 'TRV', 'TSCO', 'TSN', 'TTWO', 'TXT', 'TYL', 'UDR', 'URI', 'VFC', 'VMC', 'VRSN', 'VZ', 'WAT', 'WBA', 'WDC', 'WEC', 'WHR', 'WM', 'WMB', 'WRB', 'WST', 'WYNN', 'XEL', 'YUM', 'ZBH', 'ZION'. 

We investigate the efficiency of the collection of potential assets $S_{K}$ for portfolio formation obtained by the solution of equation (\ref{eqn:opts}). We employ the Sharpe ratio optimization and equally weighted portfolio for portfolio formation for a time period $T$, which is considered a year $y\in\{2006,2007, \hdots, 2020\}$, and the performance of the constructed portfolio is evaluated for the year $y+1$ over the full universe and the reduced universe $S_{K}.$ Maximizing Sharpe ratio subject to the no short-selling is defined as 
\begin{equation}\label{eqn:SharpMax}
    \max_w \frac{w^\dagger\widehat{\mu}}{\sqrt{w^\dagger\widehat{\Sigma}w}}
\end{equation} Here $\widehat{\mu}$ denots the vector of means  and $\widehat{\Sigma}$ is the covariance matrix of daily log returns of the concerned assets, respectively. Note that, maximizing Sharpe ratio is a widely used equivalent efficient frontier selection rule (it picks the tangency portfolio)  and under no-short constraints it stays a convex-feasible but nonlinear problem. We employ SciPy's \textit{Sequential Least Squares Programming} (SLSQP) for solving equation (\ref{eqn:SharpMax}). For equally weighted portfolios, the weights corresponding to each asset is given by $w_j=1/|S|,$ where $S$ is the set of all assets for full universe, and $S=S_K$ for the reduced universe.

\begin{figure}[h]
    \centering
    \subfigure[]{\includegraphics[width=0.45\textwidth]{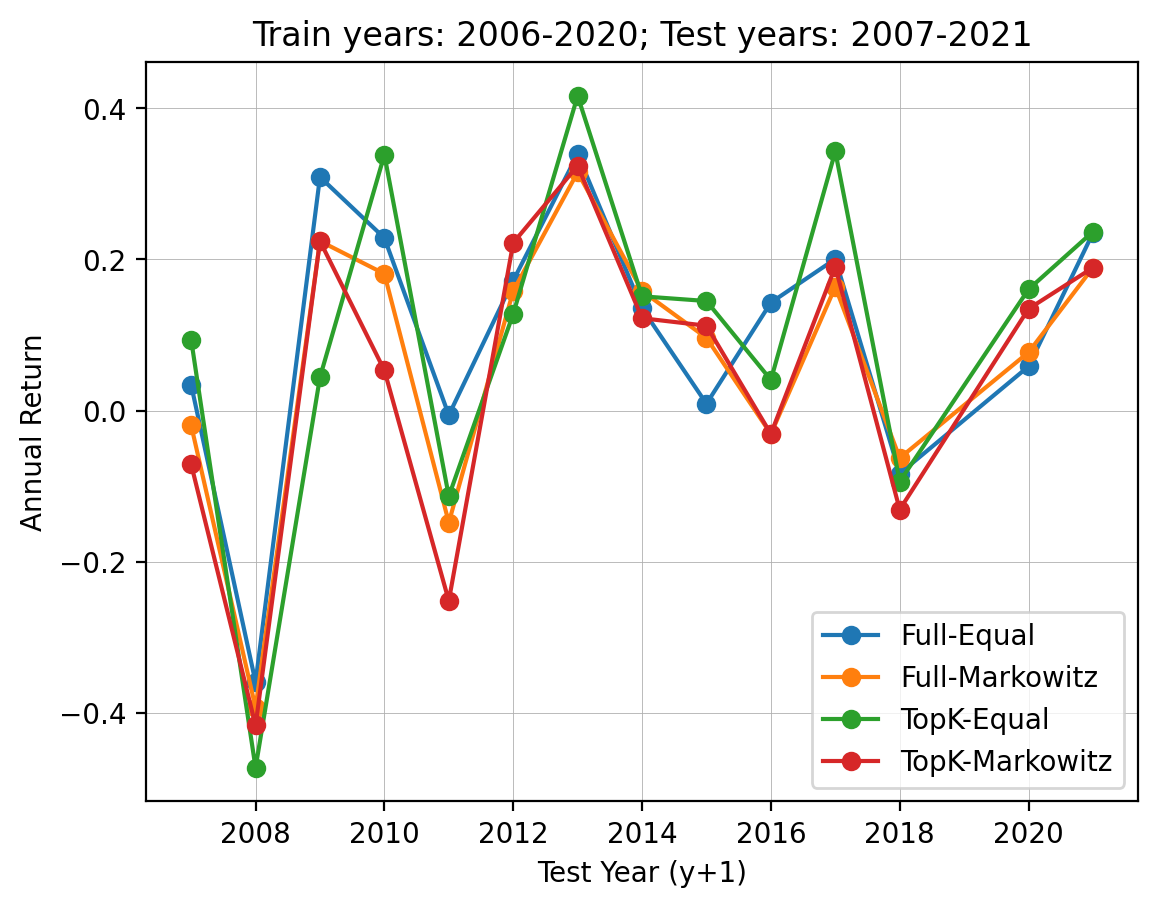}}
    ~ ~ ~
    \subfigure[]{\includegraphics[width=0.45\textwidth]{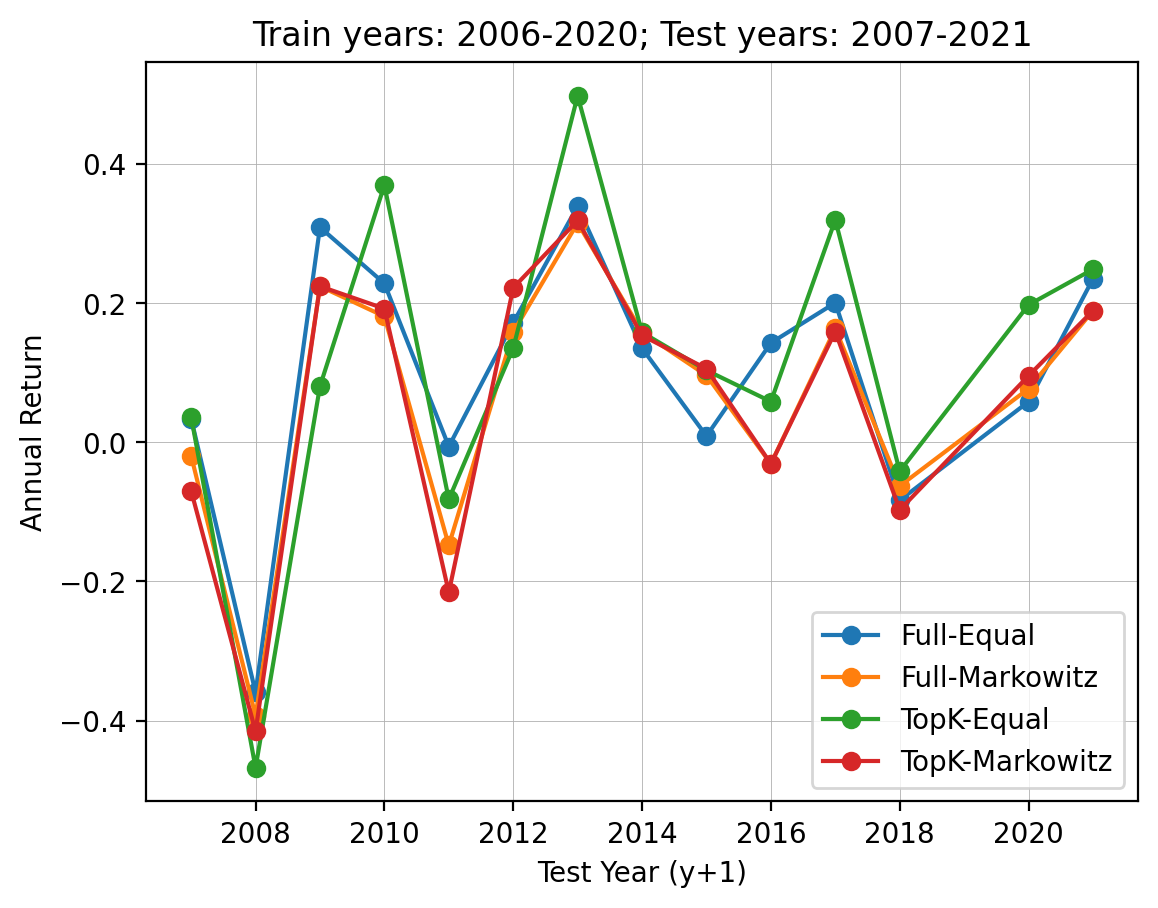}} 
    ~~~
    \subfigure[]{\includegraphics[width=0.45\textwidth]{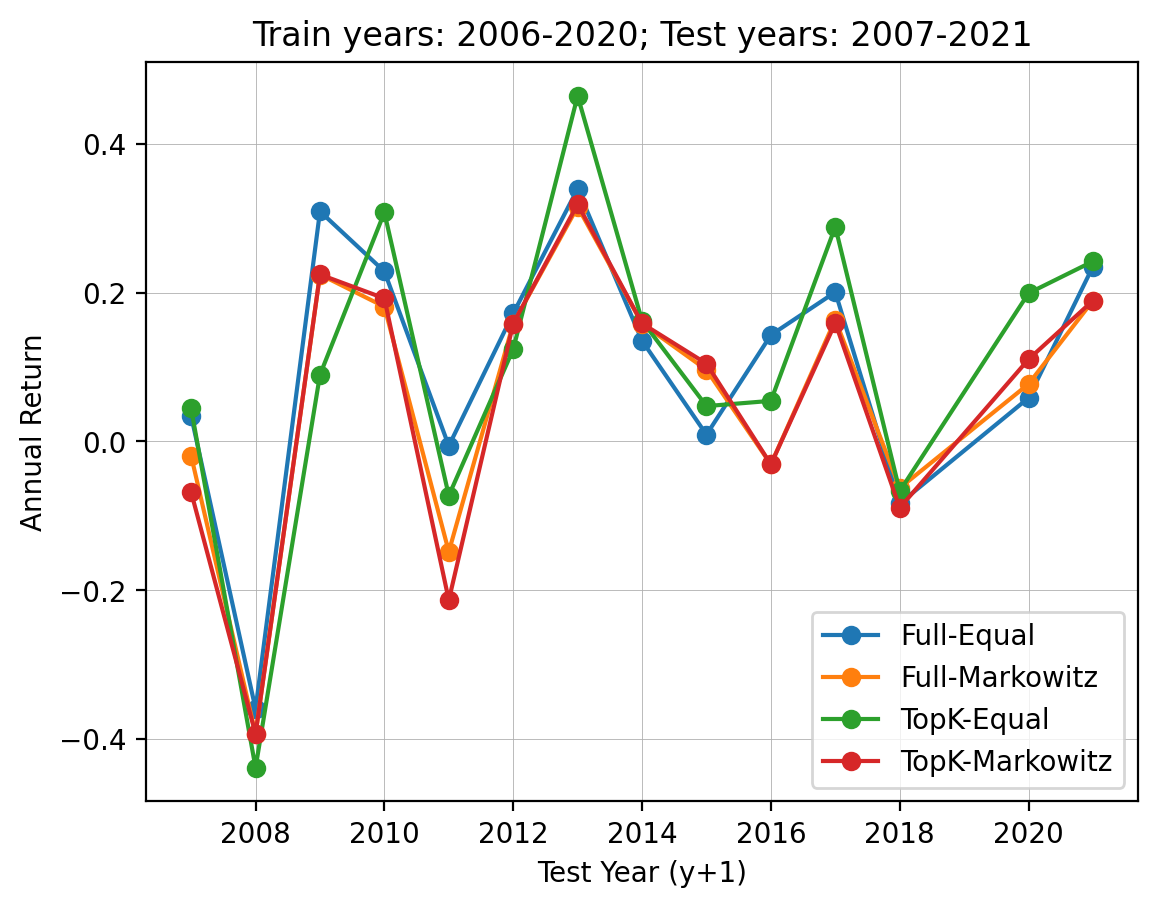}}
    ~~~
    \subfigure[]{\includegraphics[width=0.45\textwidth]{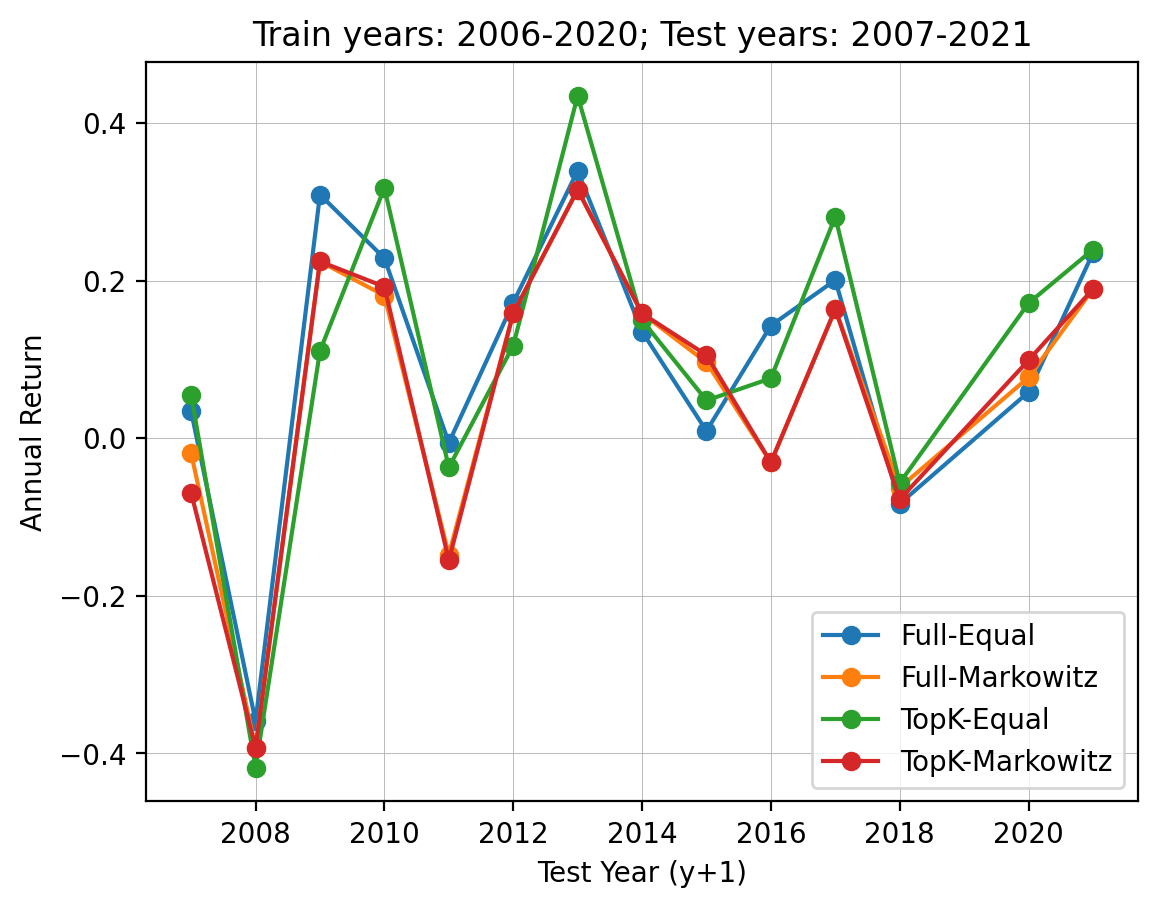}}
    \caption{Annual Return: (a) $K=20$, (b) $K=30$, (c) $K=40$, (d) $K=50.$}
    \label{fig:annualreturn}
\end{figure}

For backtesting on the out-of-sample data, we consider three metrics: annual return, annual volatility, and Sharpe value for comparing the portfolios. The daily log returns for each asset $j$ at a local time $t$ (a day) is computed as $r_{j,t}=\log(P_{j,t}/P_{j,t-1})$, where $P_{i,t}$ is the price of asset $j$ at time $t.$ Then for a given weight vector $w,$ the daily portfolio return is given by $r_{p,t}=\sum_{j=1}^n w_jr_{j,t},$ where $n$ is the number of assets in the portfolio. Then the annual return is computed as $\exp(252\overline{r}_p)-1,$ where $\overline{r}_p=\frac{1}{T}\sum_{t=1}^T r_{p,t}.$ The sample standard deviation of the daily log returns is given by $\sigma_p^{(daily)}=\sqrt{\frac{1}{T-1}\sum_{t=1}^T (r_{p,t}-\overline{r}_p)^2},$ and consequently the annual volatility is computed as $\sigma_p^{(ann)}=\sigma_p^{(daily)}\sqrt{252}.$ Finally, the Sharpe value for a portfolio is considered as the ratio of annual return and $\sigma_p^{(ann)}.$

In order to test the performance of our dimensionality reduction technique, we determine $S_{K}$ for $K\in\{20,30,40,50\}.$ In Tables \ref{tab:RU16-20}, \ref{tab:RU11-15} and \ref{tab:RU06-10}, we include assets in $S_{50}$ for all the years $2006, 2007, \hdots, 2020$ as obtained by executing Algorithm \ref{Alg:dr_omv}. The figures in Figure \ref{fig:annualreturn}, \ref{fig:annualV} and \ref{fig:SR} exhibit the annual return, annual volatility, and Sharpe values for the portfolios respectively, obtained by Sharpe value maximization (Full-Markowitz), Equally weighted portfolio (Full-Equal), Sharpe value maximization of reduced universe $S_{K}$ (TopK-Markowitz), Equally weighted portfolio of the reduced universe $S_{K}$ (TopK-Equal).

Obviously, annual return measures the profitability, and hence higher values are better. From Figure \ref{fig:annualreturn}, it is clear that the TopK-Equal portfolio and TopK-Markowitz obtained by the proposed OMV dimension reduction technique performs moderately better than the full universe across almost all the 15 years (2006 - 2020) except the year 2009 for all choices of $K$. The annual volatility, the standard deviation of annualized returns measure how uncertain the portfolio's return are, hence lower is better. It can be seen from Figure \ref{fig:annualV} that the obtained portfolios based on reduced universe are volatile across all years except the year $2009.$ Finally in Figure \ref{fig:SR}, it can be observed that the results are mixed, and on average the Sharpe value of equally weighted portfolios, either the TopK or the full universe provide the highest values. Based on these results, we could conclude that the OMV reduced portfolio for this dataset forms a highly volatile profitable portfolios compared to the full universe portfolio optimization, however provide moderately better annual return across all 15 years of experiment. 

\begin{figure}[h]
    \centering
    \subfigure[]{\includegraphics[width=0.45\textwidth]{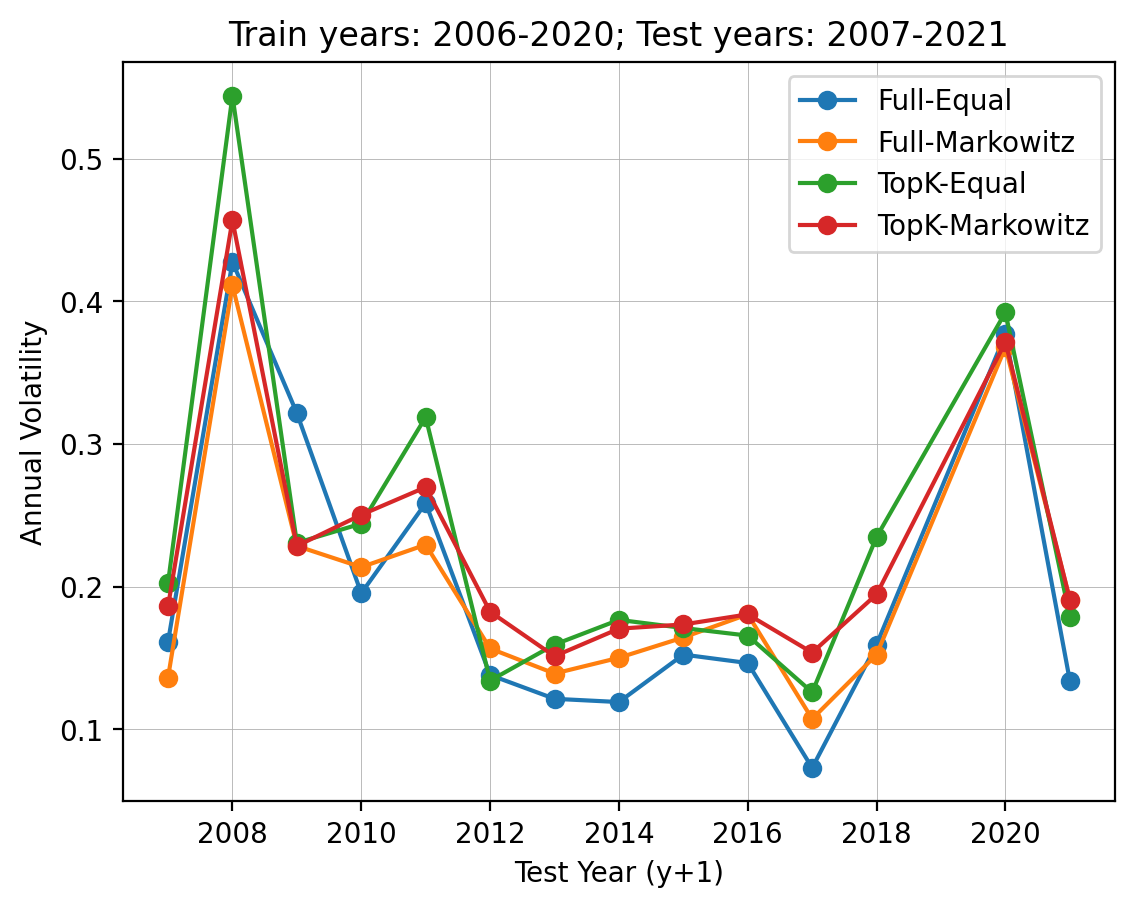}}
    ~ ~ ~
    \subfigure[]{\includegraphics[width=0.45\textwidth]{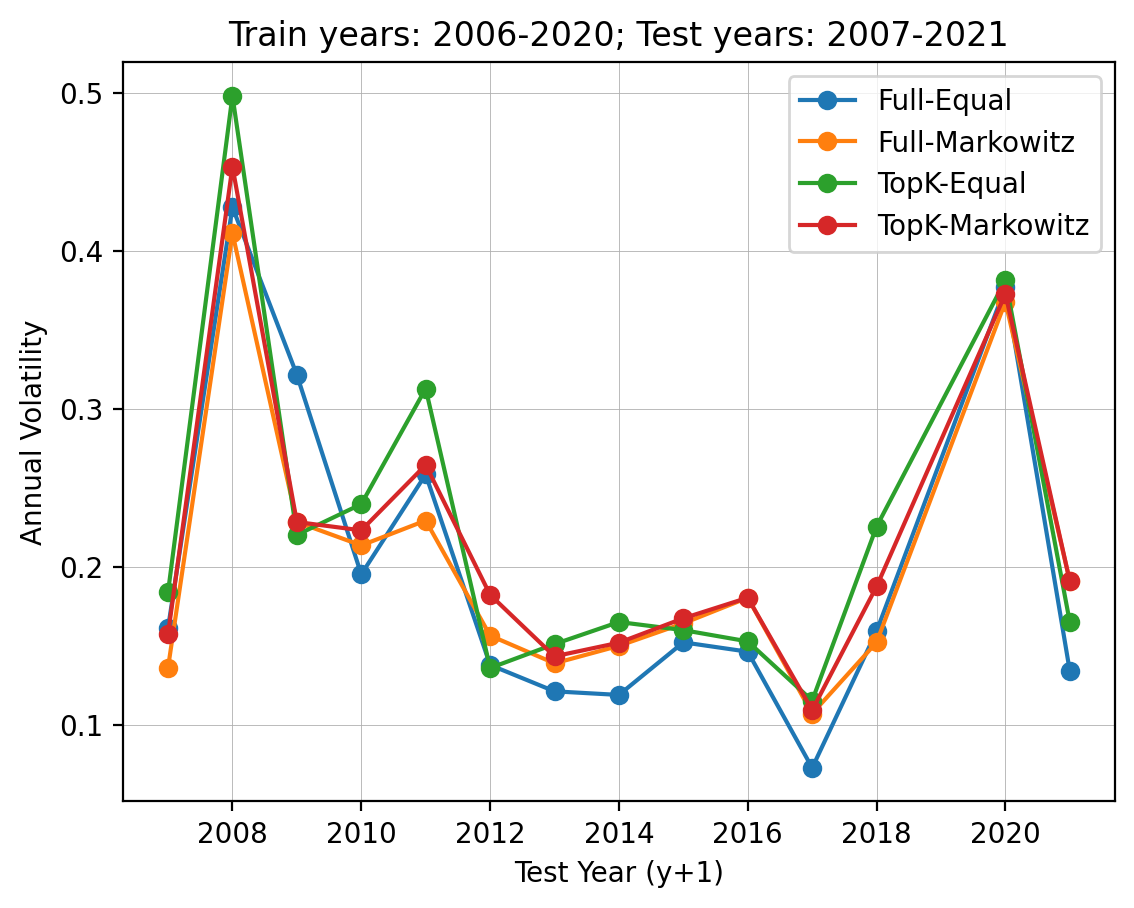}} 
    ~~~
    \subfigure[]{\includegraphics[width=0.45\textwidth]{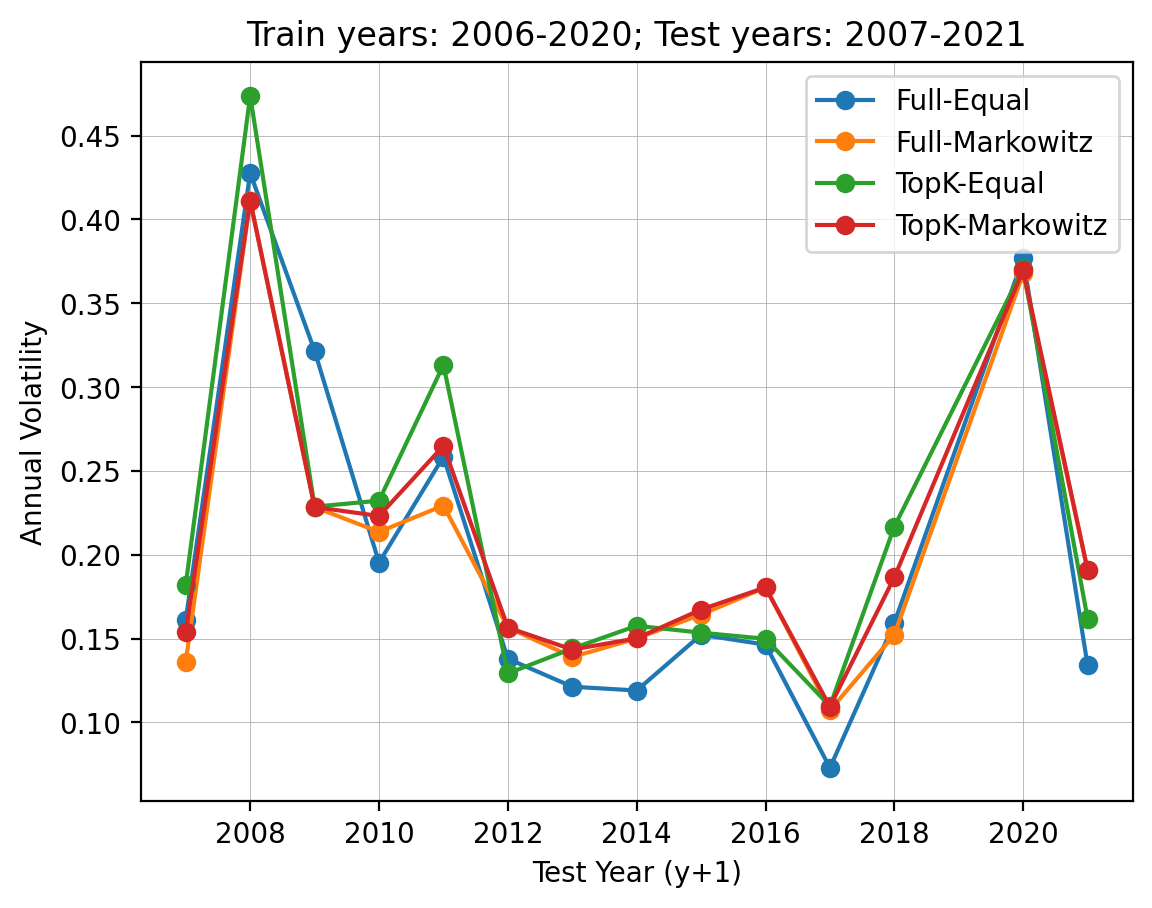}}
    ~~~
    \subfigure[]{\includegraphics[width=0.45\textwidth]{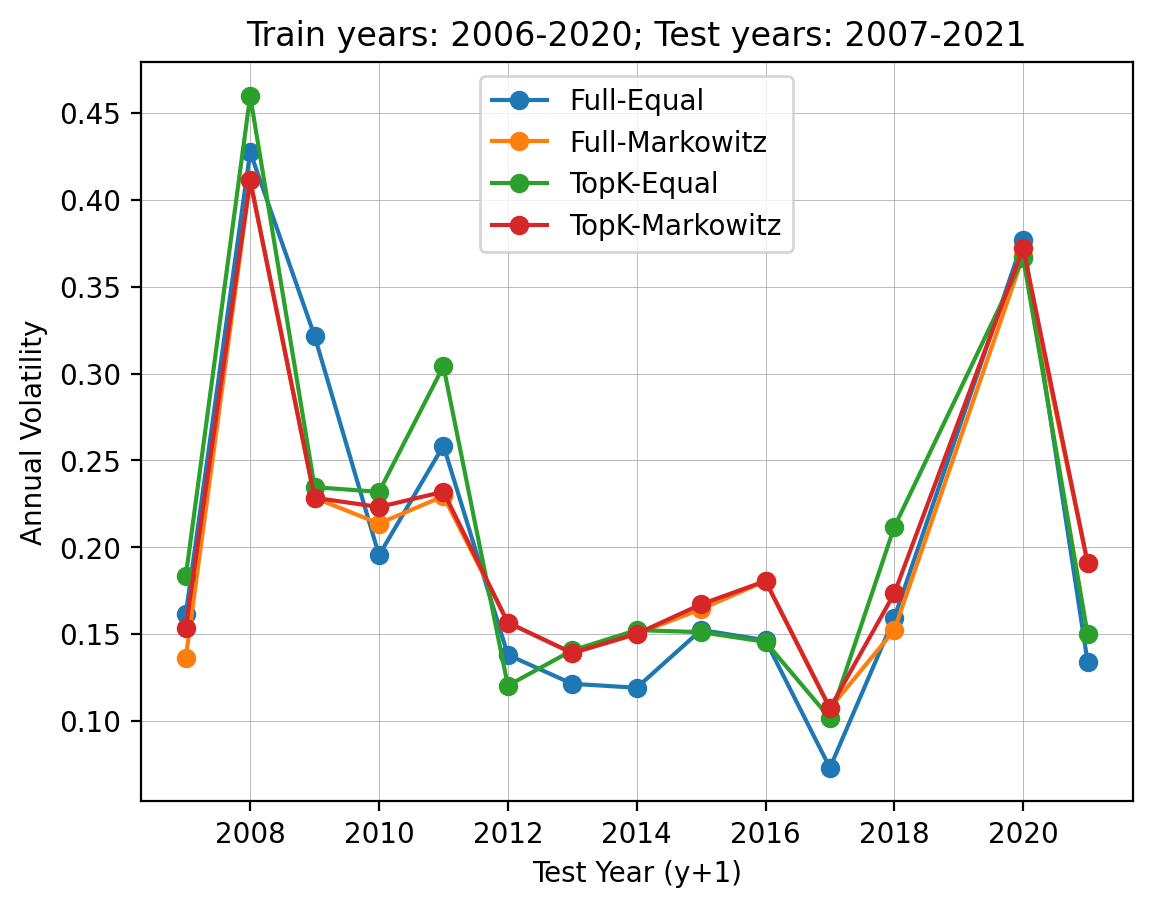}}
    \caption{Annual Volatility: (a) $K=20$, (b) $K=30$, (c) $K=40$, (d) $K=50.$}
    \label{fig:annualV}
\end{figure}

\begin{figure}[h]
    \centering
    \subfigure[]{\includegraphics[width=0.45\textwidth]{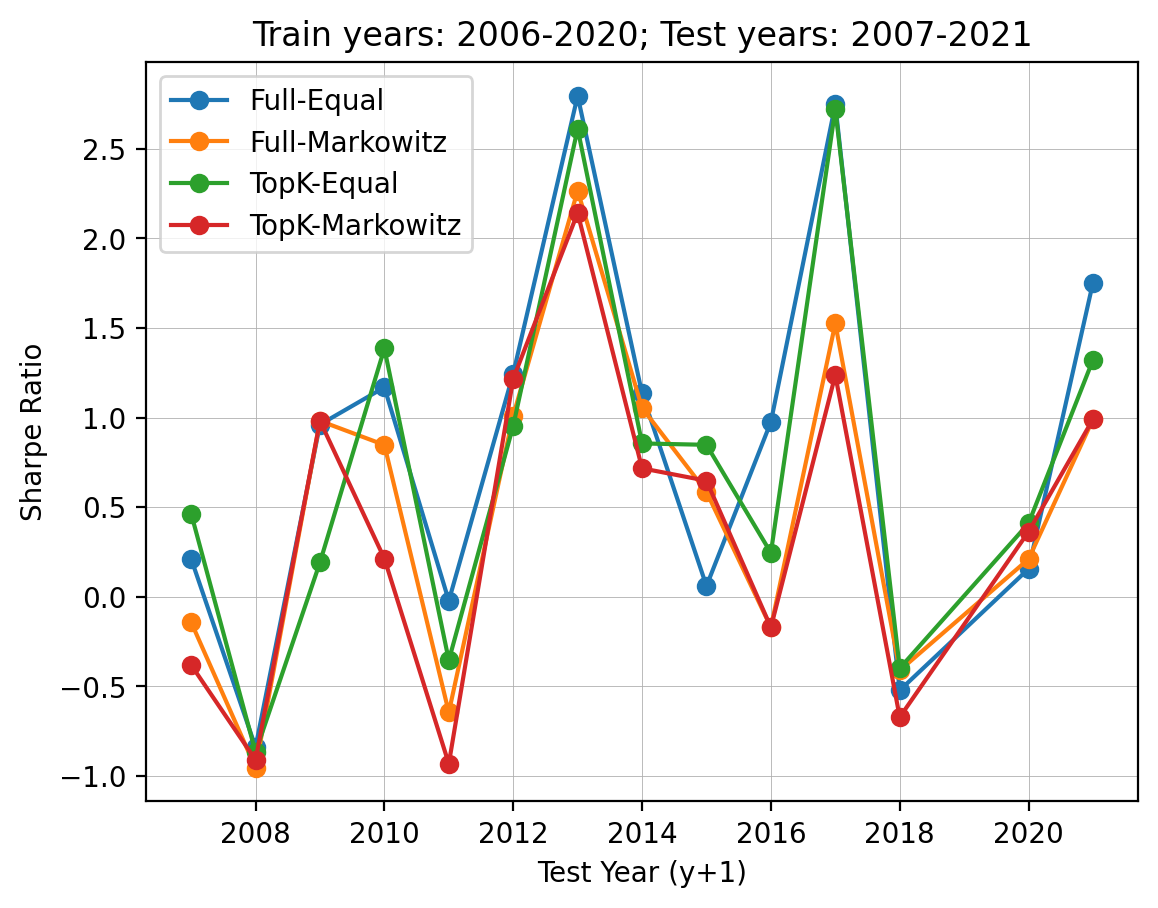}}
    ~ ~ ~
    \subfigure[]{\includegraphics[width=0.45\textwidth]{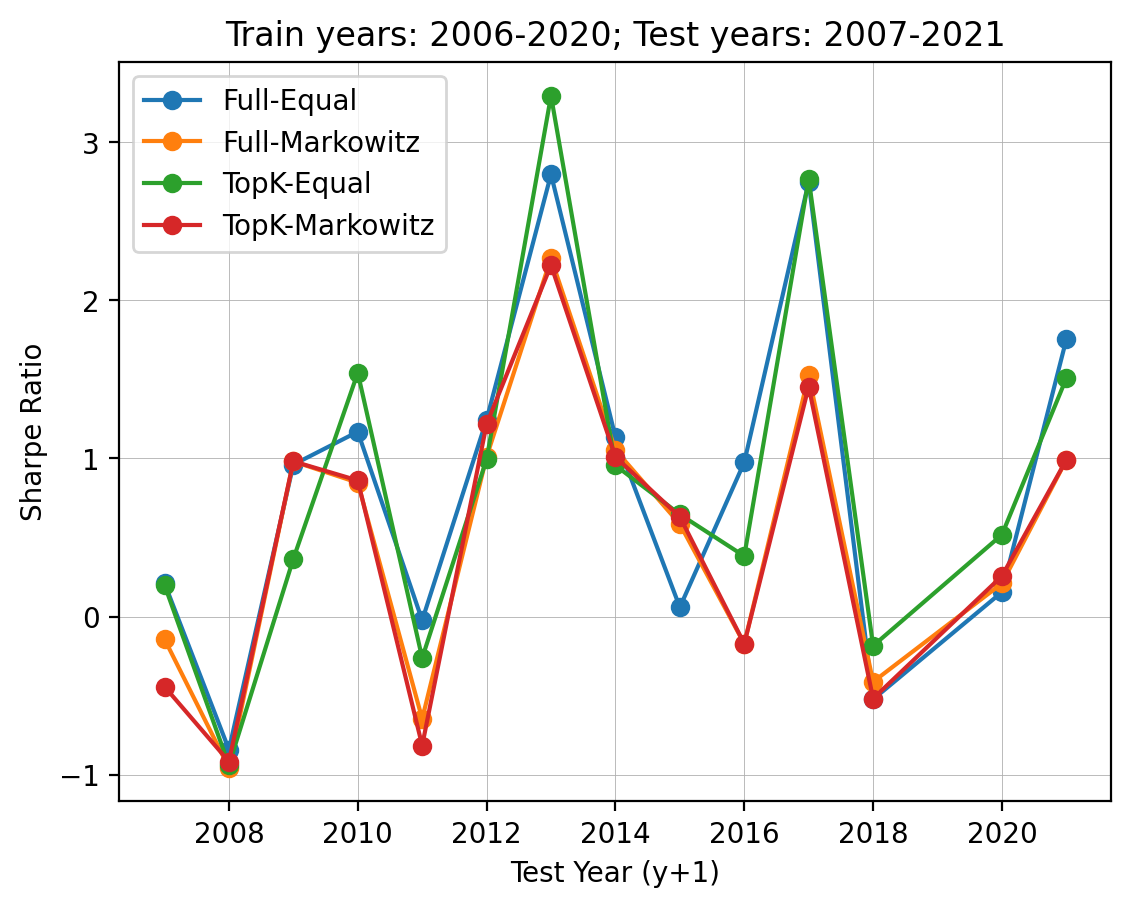}} 
    ~~~
    \subfigure[]{\includegraphics[width=0.45\textwidth]{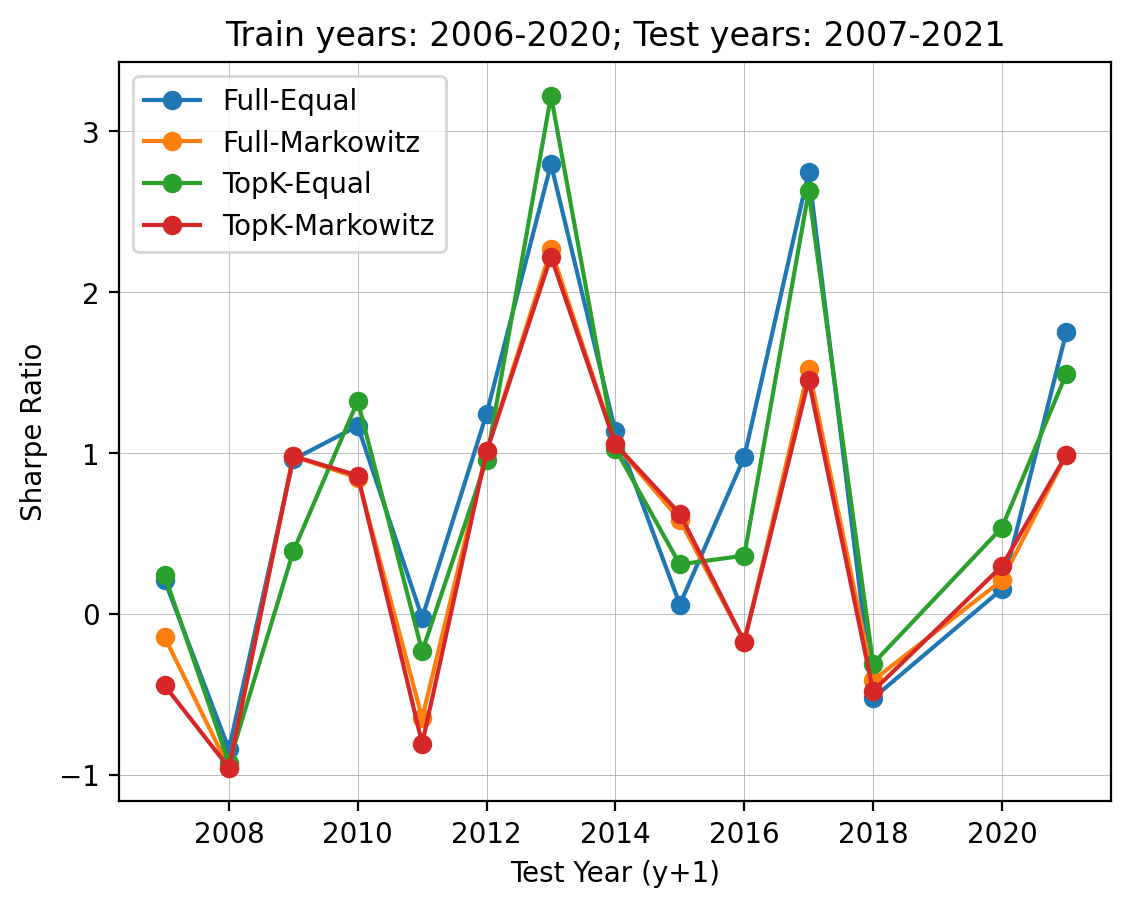}}
    ~~~
    \subfigure[]{\includegraphics[width=0.45\textwidth]{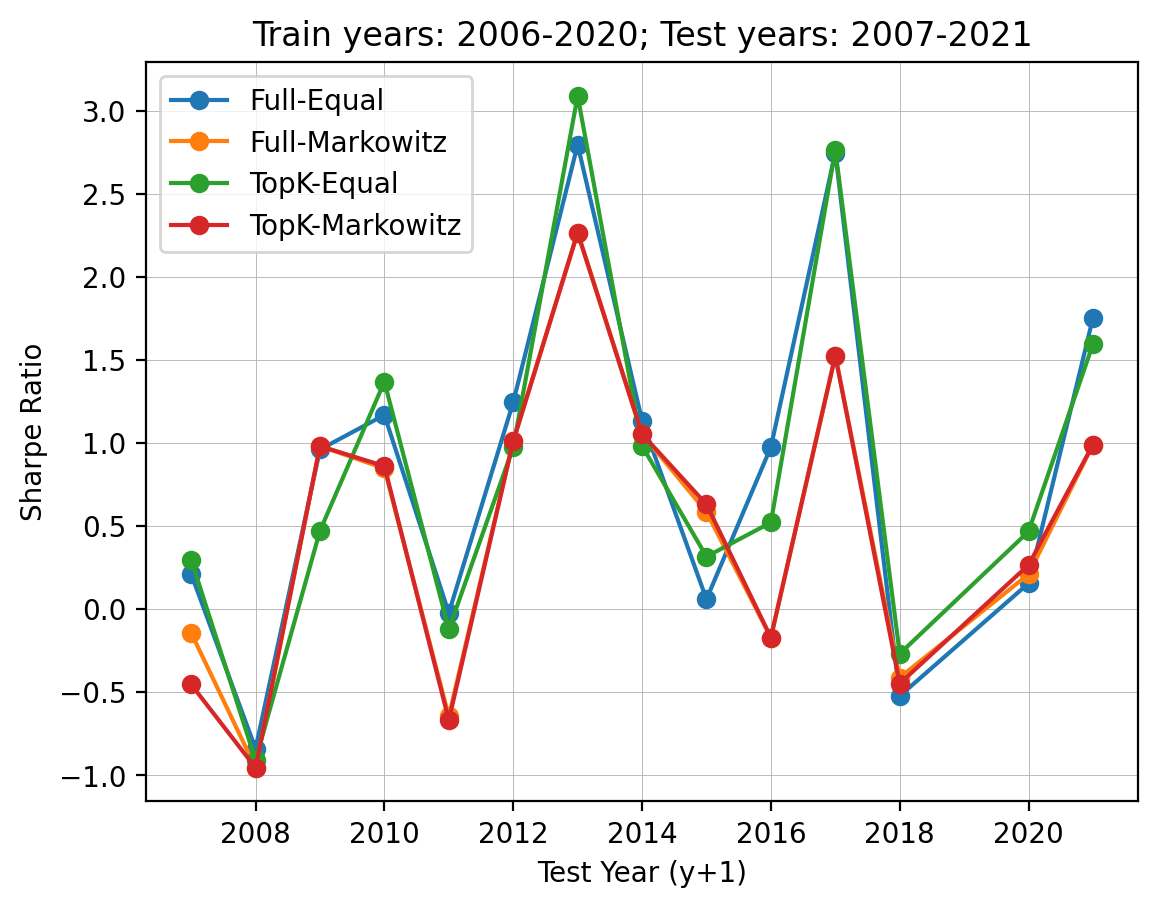}}
    \caption{Sharpe ratio:  (a) $K=20$, (b) $K=30$, (c) $K=40$, (d) $K=50.$}
    \label{fig:SR}
\end{figure}

\begin{table}[htp]
    \centering
    \begin{tabular}{c|c|c}
      Train year & Test year & $S_{50}$  \\
      \hline \hline
    2020 & 2021 &  \begin{tiny}'AMD', 'WST', 'ABMD', 'ALB', 'IDXX', 'AMZN', 'TTWO', 'NFLX', 'SNPS', 'ALGN', \end{tiny}\\
    && \begin{tiny} 'ROL', 'ATVI', 'ADSK', 'BIO', 'TSCO', 'DVA', 'ODFL', 'TYL', 'PKI', 'SIVB',\end{tiny} \\
    && \begin{tiny}'TMO', 'EBAY', 'CLX', 'AMAT', 'MSFT', 'EA', 'DPZ', 'CRM', 'CPRT', 'A', \end{tiny} \\
    && \begin{tiny}'GGG', 'EFX', 'ISRG', 'URI', 'CDE', 'LOW', 'GOOG', 'MCHP', 'FAST', 'CTSH', \end{tiny}\\
    && \begin{tiny}'CHD', 'NEE', 'AJG', 'PH', 'EL', 'CTAS', 'ABT', 'SHW', 'AKAM', 'CMI'\end{tiny} \\
    \hline
    2019 & 2020 &   \begin{tiny}'AMD', 'CDE', 'CPRT', 'AMAT', 'TSN', 'GPN', 'WDC', 'HES', 'MLM', 'NVR',\end{tiny} \\
    && \begin{tiny}
        'BBY', 'CAG', 'MCO', 'TYL', 'EL', 'FISV', 'CPB', 'DOV', 'SNPS', 'FMC', \end{tiny} \\
    && \begin{tiny}
        'URI', 'PHM', 'BIO', 'ODFL', 'CTAS', 'AMT', 'SO', 'PLD', 'EW', 'WST',    \end{tiny} \\
    && \begin{tiny}
        'MSFT', 'VMC', 'DVA', 'NEE', 'APD', 'SHW', 'ACN', 'HSY', 'EFX', 'VFC',\end{tiny} \\
    && \begin{tiny}
        'SRE', 'AVY', 'ZBH', 'BAC', 'NOC', 'GIS', 'KIM', 'MAA', 'ADSK', 'IDXX' \end{tiny} \\
        \hline 
    2018 & 2019 &  \begin{tiny}
        'AMD', 'ABMD', 'AAP', 'ORLY', 'NFLX', 'BSX', 'DPZ', 'CHD', 'ILMN', 'EW',
    \end{tiny}\\
    && \begin{tiny}
        'CRM', 'AMZN', 'ISRG', 'VRSN', 'KSS', 'ABT', 'ADSK', 'GWW', 'HRL', 'TJX',
    \end{tiny} \\
    && \begin{tiny}
         'IDXX', 'MSFT', 'COO', 'ROL', 'TMO', 'AJG', 'HUM', 'COP', 'CNC', 'RLI',
    \end{tiny} \\
    && \begin{tiny}
         'MOS', 'NEE', 'TSCO', 'CSCO', 'YUM', 'AMT', 'AMGN', 'SBUX',  'O', 'ADP',
    \end{tiny} \\
    && \begin{tiny}
        'FISV', 'CPRT', 'ECL', 'FIS', 'JKHY', 'NTAP', 'CTAS', 'VZ', 'PKI', 'WEC'
    \end{tiny} \\
   \hline
   2017 &2018 &  \begin{tiny}
       'ALGN', 'TTWO', 'NVR', 'WYNN', 'CNC', 'PHM', 'ATVI', 'ILMN', 'ISRG', 'EL',
   \end{tiny}\\
   && \begin{tiny}
       'FMC', 'ABMD', 'BBY', 'MAR', 'AVY', 'URI', 'AMZN', 'AMAT', 'NTAP', 'GGG',
   \end{tiny} \\
   && \begin{tiny}
       'NFLX', 'PVH', 'MCO', 'CPRT', 'ODFL', 'VRSN', 'SHW', 'CRM', 'ABT', 'BAX',
   \end{tiny} \\
   && \begin{tiny}
       'ALB', 'GPN', 'AME', 'A', 'SWK', 'SNPS', 'RCL', 'WAT', 'ROK', 'HD',
   \end{tiny} \\
   && \begin{tiny}
        'ADSK', 'VFC', 'PKG', 'PH', 'PKI', 'ROP', 'ROL', 'AMT', 'MSFT', 'MCHP'
   \end{tiny} \\
  \hline
   2016 & 2017 & \begin{tiny}
       'CDE', 'AMD', 'AMAT', 'IDXX', 'MLM', 'ZION', 'CMI', 'ALB', 'RF', 'APA',
   \end{tiny} \\
   && \begin{tiny}
       'ODFL', 'DPZ', 'SIVB', 'URI', 'TTWO', 'ALGN', 'BBY', 'FMC', 'CPRT',
   \end{tiny} \\
   && \begin{tiny}
       'PH', 'WST', 'MCHP', 'WM', 'PKG', 'FITB', 'JBHT', 'PCAR', 'NTAP', 'BAC',
   \end{tiny} \\
   && \begin{tiny}
       'BIO', 'HES', 'COO', 'VMC', 'ITW', 'LNC', 'ROK', 'ROL', 'SYK', 'ADI',
   \end{tiny} \\
   && \begin{tiny}
       'JPM', 'WYNN', 'DGX', 'AKAM', 'PRU', 'ABMD', 'SNPS', 'T', 'ARE', 'AJG', 'NOC'
   \end{tiny} \\
   \hline 
    \end{tabular}
    \caption{Reduced universes for backtesting: For year $y\in\{2020,2019, 2018, 2017, 2016\},$ $S_{50}$ is obtained using Algorithm \ref{Alg:dr_omv}. }
    \label{tab:RU16-20}
\end{table}

\begin{table}[htp]
    \centering
    \begin{tabular}{c|c|c}
      Train year & Test year & $S_{50}$  \\
      \hline \hline
    2015 & 2016 &  \begin{tiny} 'NFLX', 'ABMD', 'AMZN', 'ATVI', 'GPN', 'TYL', 'EXR', 'EA', 'HRL', 'GOOG',\end{tiny}\\
    && \begin{tiny} 'VRSN', 'SBUX', 'VMC', 'ALK', 'BSX', 'TSN', 'EFX', 'AOS', 'CRM', 'ORLY',  \end{tiny} \\
    && \begin{tiny} 'NVR', 'CNC', 'CUK', 'HUM', 'TTWO', 'HAS', 'RLI', 'NOC', 'EW', 'JKHY', \end{tiny} \\
    && \begin{tiny} 'JNPR', 'FISV', 'MLM', 'RCL', 'UDR', 'MAA', 'HD', 'MDLZ', 'PKI', 'CLX',\end{tiny}\\
    && \begin{tiny} 'AVY', 'ALGN', 'CPB', 'DPZ', 'ROP', 'MSFT', 'NI', 'ROL', 'WAT', 'EBAY'\end{tiny} \\
    \hline
    2014 & 2015 &  \begin{tiny} 'EW', 'EA', 'CNC', 'RCL', 'ILMN', 'TTWO', 'ALK', 'MAR', 'ORLY', 'ABMD', \end{tiny} \\
    && \begin{tiny} 'ODFL', 'ARE', 'AAP', 'SHW', 'ISRG', 'AMGN', 'AMAT', 'HUM', 'EXR', 'REG',
         \end{tiny} \\
    && \begin{tiny} 'IDXX', 'LOW', 'RHI', 'WBA', 'LEG', 'DPZ', 'UDR', 'WEC', 'PNW', 'XEL',
           \end{tiny} \\
    && \begin{tiny} 'O', 'LNT', 'WDC', 'ES', 'COO', 'AKAM', 'NI', 'BXP', 'CHRW', 'SLG',
        \end{tiny} \\
    && \begin{tiny} 'SPG', 'URI', 'KIM', 'CTAS', 'NVR', 'DGX', 'APD', 'NOC', 'BRK-A', 'HD'
         \end{tiny} \\
        \hline 
    2013 & 2014 & \begin{tiny}
        'NFLX', 'BBY', 'ABMD', 'BSX', 'GILD', 'TYL', 'ALGN', 'ILMN', 'BIIB', 'WDC',
    \end{tiny}\\
    && \begin{tiny}
       'SEE', 'SIVB', 'LNC', 'WST', 'SCHW', 'TSCO', 'TSN', 'WYNN', 'ALK', 'ATVI',
    \end{tiny} \\
    && \begin{tiny}
         'URI', 'AMD', 'AOS', 'VFC', 'TMO', 'NOC', 'PRU', 'EA', 'PKG', 'ADM',
    \end{tiny} \\
    && \begin{tiny}
     'TTWO', 'HUM', 'GOOG', 'WBA', 'HRB', 'AMZN', 'DPZ', 'HES', 'AMAT', 'FIS',    
    \end{tiny} \\
    && \begin{tiny}
        'FLS', 'AAP', 'WHR', 'ODFL', 'BWA', 'VRSN', 'HAS', 'TJX', 'MCO', 'NEOG'
    \end{tiny} \\
     \hline
   2012 &2013 &  \begin{tiny}
    'PHM', 'WHR', 'BAC', 'ILMN', 'GILD', 'SHW', 'EBAY', 'CRM', 'EMN', 'CCI',  
   \end{tiny}\\
   && \begin{tiny}
       'TYL', 'RF', 'PPG', 'PVH', 'AOS', 'URI', 'EXR', 'PKG', 'PKI', 'NEOG',
   \end{tiny} \\
   && \begin{tiny}
       'DVA', 'HD', 'MCO', 'AMZN', 'WST', 'LOW', 'FLS', 'NFLX', 'NWL', 'WDC',
   \end{tiny} \\
   && \begin{tiny}
     'EFX', 'NVR', 'TMO', 'RCL', 'AMGN', 'JBHT', 'BIIB', 'BAX', 'TJX', 'FMC', 
   \end{tiny} \\
   && \begin{tiny}
        'TXT', 'COO', 'JPM', 'FIS', 'AME', 'AMT', 'FISV', 'SRE', 'EW', 'LNC'
   \end{tiny} \\
   \hline
   2011 & 2012 &  \begin{tiny}
       'DPZ', 'ABMD', 'ISRG', 'BIIB', 'HUM', 'CNC', 'VFC', 'TJX', 'SBUX', 'TSCO',
   \end{tiny} \\
   && \begin{tiny}
      'EL', 'TYL', 'FAST', 'HRB', 'RLI', 'HSY', 'ORLY', 'NI', 'EXR', 'CHD', 
   \end{tiny} \\
   && \begin{tiny}
       'ALK', 'WMB', 'GWW', 'URI', 'EA', 'COO', 'SPG', 'CPRT', 'D', 'ODFL',
   \end{tiny} \\
   && \begin{tiny}
    'MCO', 'SO', 'WRB', 'YUM', 'CTAS', 'ALGN', 'TSN', 'MDLZ', 'AMGN', 'ABT',  
   \end{tiny} \\
   && \begin{tiny}
       'HD', 'LNT', 'CAG', 'WEC', 'AMT', 'XEL', 'NEE', 'KMB', 'FTI', 'GPC'
   \end{tiny} \\
   \hline 
    \end{tabular}
    \caption{Reduced universes for backtesting: For year $y\in\{2015,2014, 2013, 2012, 2011\}$, $S_{50}$ is obtained using Algorithm \ref{Alg:dr_omv}. }
    \label{tab:RU11-15}
\end{table}

\begin{table}[htp]
    \centering
    \begin{tabular}{c|c|c}
      Train year & Test year & $S_{50}$  \\
      \hline \hline
    2010 & 2011 &  \begin{tiny} 'NFLX', 'URI', 'CMI', 'ILMN', 'BWA', 'EW', 'DPZ', 'ZION', 'AKAM', 'HBAN',\end{tiny}\\
    && \begin{tiny}  'TSCO', 'CRM', 'RCL', 'WYNN', 'NEOG', 'AAP', 'ALK', 'EL', 'NTAP', 'ODFL', \end{tiny} \\
    && \begin{tiny} 'ORLY', 'CDE', 'CTSH', 'PH', 'COO', 'FITB', 'PVH', 'PCAR', 'FTI', 'NOV', \end{tiny} \\
    && \begin{tiny} 'ROL', 'AME', 'ALB', 'HAS', 'EXR', 'ADSK', 'ROK', 'MAR', 'HST', 'TSN',\end{tiny}\\
    && \begin{tiny} 'UDR', 'ROP', 'FAST', 'FMC', 'PPG', 'YUM', 'SBUX', 'GWW', 'HRL', 'RF'\end{tiny} \\
    \hline
    2009 & 2010 &  \begin{tiny} 'AMD', 'WDC', 'AMZN', 'NTAP', 'CTSH', 'SBUX', 'COO', 'CRM', 'ISRG', 'FTI', \end{tiny} \\
    && \begin{tiny} 'CCI', 'CDE', 'ALGN', 'NFLX', 'PVH', 'GOOG', 'A', 'MOS', 'WHR', 'DPZ',
         \end{tiny} \\
    && \begin{tiny} 'SLG', 'TJX', 'NOV', 'EMN', 'RCL', 'WAT', 'FLS', 'TYL', 'AKAM', 'EW',
           \end{tiny} \\
    && \begin{tiny} 'GPN', 'PKG', 'ADI', 'CMI', 'EBAY', 'IVZ', 'JCI', 'MSFT', 'BEN', 'NVR',
        \end{tiny} \\
    && \begin{tiny} 'PRU', 'SIVB', 'WBA', 'ALB', 'APD', 'HST', 'KSS', 'JNPR', 'EL', 'BWA'
         \end{tiny} \\
        \hline 
    2008 & 2009 &  \begin{tiny}
        'AMGN', 'HRB', 'EW', 'ODFL', 'ALK', 'HAS', 'NFLX', 'GILD', 'RLI', 'GIS',
    \end{tiny}\\
    && \begin{tiny}
       'AJG', 'ABMD', 'WRB', 'CHD', 'PHM', 'SHW', 'CHRW', 'WM', 'TSCO', 'DGX',
    \end{tiny} \\
    && \begin{tiny}
         'JBHT', 'LOW', 'SO', 'ORLY', 'ABT', 'ROL', 'NEOG', 'WST', 'TYL', 'BAX',
    \end{tiny} \\
    && \begin{tiny}
         'MMC', 'ACN', 'GWW', 'ADP', 'AAP', 'LEG', 'VMC', 'HSY', 'MAA', 'FAST',
    \end{tiny} \\
    && \begin{tiny}
      'ILMN', 'HD', 'NVR', 'O', 'DVA', 'CB', 'CLX', 'AOS', 'WEC', 'TRV'  
    \end{tiny} \\
   \hline
   2007 &2008 &  \begin{tiny}
      'MOS', 'ISRG', 'NOV', 'AMZN', 'CMI', 'NEOG', 'HES', 'FTI', 'FLS', 'ATVI',
   \end{tiny}\\
   && \begin{tiny}
      'CRM', 'JNPR', 'WAT', 'BWA', 'VRSN', 'APA', 'IDXX', 'SLB', 'OXY', 'ILMN', 
   \end{tiny} \\
   && \begin{tiny}
       'GOOG', 'WDC', 'TXT', 'PH', 'ADM', 'CPRT', 'AME', 'FMC', 'GILD', 'HUM',
   \end{tiny} \\
   && \begin{tiny}
      'APD', 'WMB', 'BRK-A', 'SYK', 'JBHT', 'SCHW', 'CHRW', 'CCI', 'YUM', 'ROL', 
   \end{tiny} \\
   && \begin{tiny}
        'PKG', 'IVZ', 'TMO', 'CHD', 'MLM', 'BIO', 'BAX', 'NEE', 'PCAR', 'NTRS'
   \end{tiny} \\
    \hline
   2006 & 2007 &  \begin{tiny}
       'ILMN', 'AKAM', 'ALGN', 'WST', 'ALB', 'SLG', 'WYNN', 'TYL', 'IVZ', 'CSCO',
   \end{tiny} \\
   && \begin{tiny}
       'PVH', 'ABMD', 'BXP', 'VFC', 'TMO', 'CTSH', 'MOS', 'T', 'IFF', 'FTI',
   \end{tiny} \\
   && \begin{tiny}
       'NTAP', 'ES', 'KSS', 'FMC', 'MAR', 'LH', 'SHW', 'CAG', 'LNT', 'AMT',
   \end{tiny} \\
   && \begin{tiny}
       'PCAR', 'KIM', 'HAS', 'MLM', 'NEE', 'UDR', 'CPB', 'ADM', 'SLB', 'SNPS',
   \end{tiny} \\
   && \begin{tiny}
      'REG', 'SPG', 'VMC', 'SCHW', 'CPRT', 'MDLZ', 'CMI', 'WAT', 'ACN', 'HST'
   \end{tiny} \\
   \hline  
    \end{tabular}
        \caption{Reduced universes for backtesting: For year $y\in\{2010,2009, 2008, 2007, 2006\}$, $S_{50}$ is obtained using Algorithm \ref{Alg:dr_omv}.}
    \label{tab:RU06-10}
\end{table}

In the next section, we compare the out-of-sample performance of the hedge score based selection of $K$ assets with alternative methods for selecting $K$ assets in portfolio construction.

\begin{figure}[htp]
    \centering
    \subfigure[]{\includegraphics[width=0.40\textwidth]{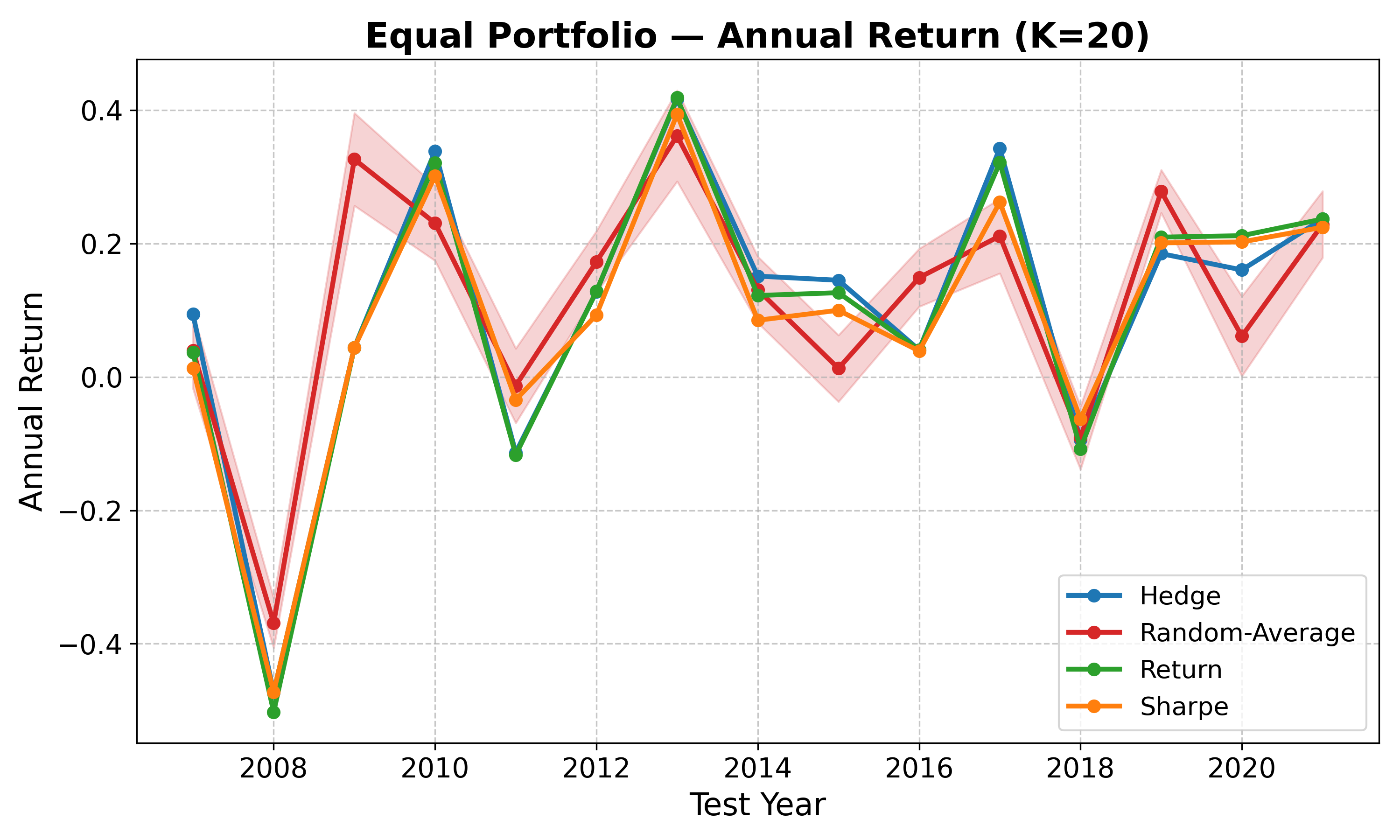}}
    ~ ~ ~
    \subfigure[]{\includegraphics[width=0.40\textwidth]{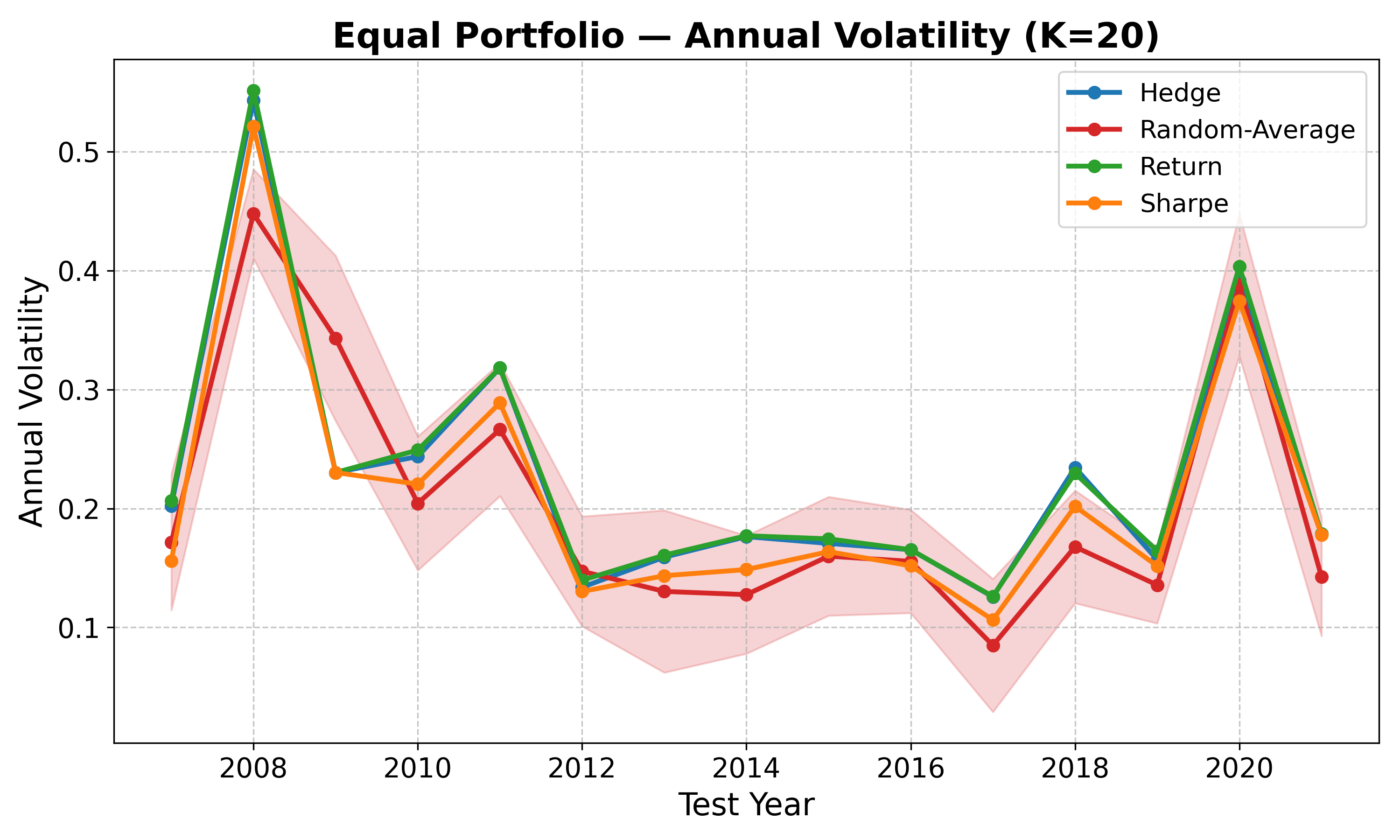}}
    ~ ~ ~
    \subfigure[]{\includegraphics[width=0.40\textwidth]{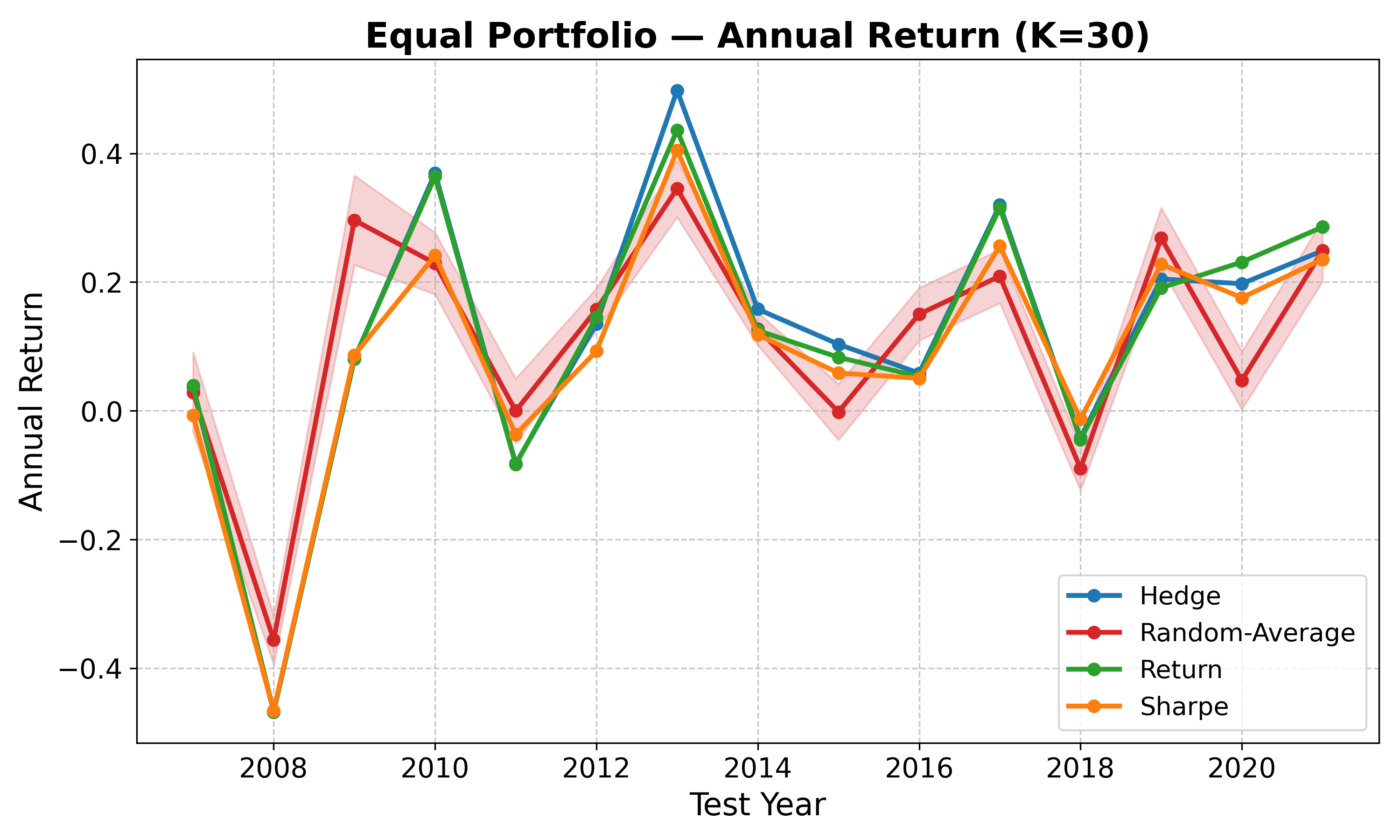}} 
    ~~~
    \subfigure[]{\includegraphics[width=0.40\textwidth]{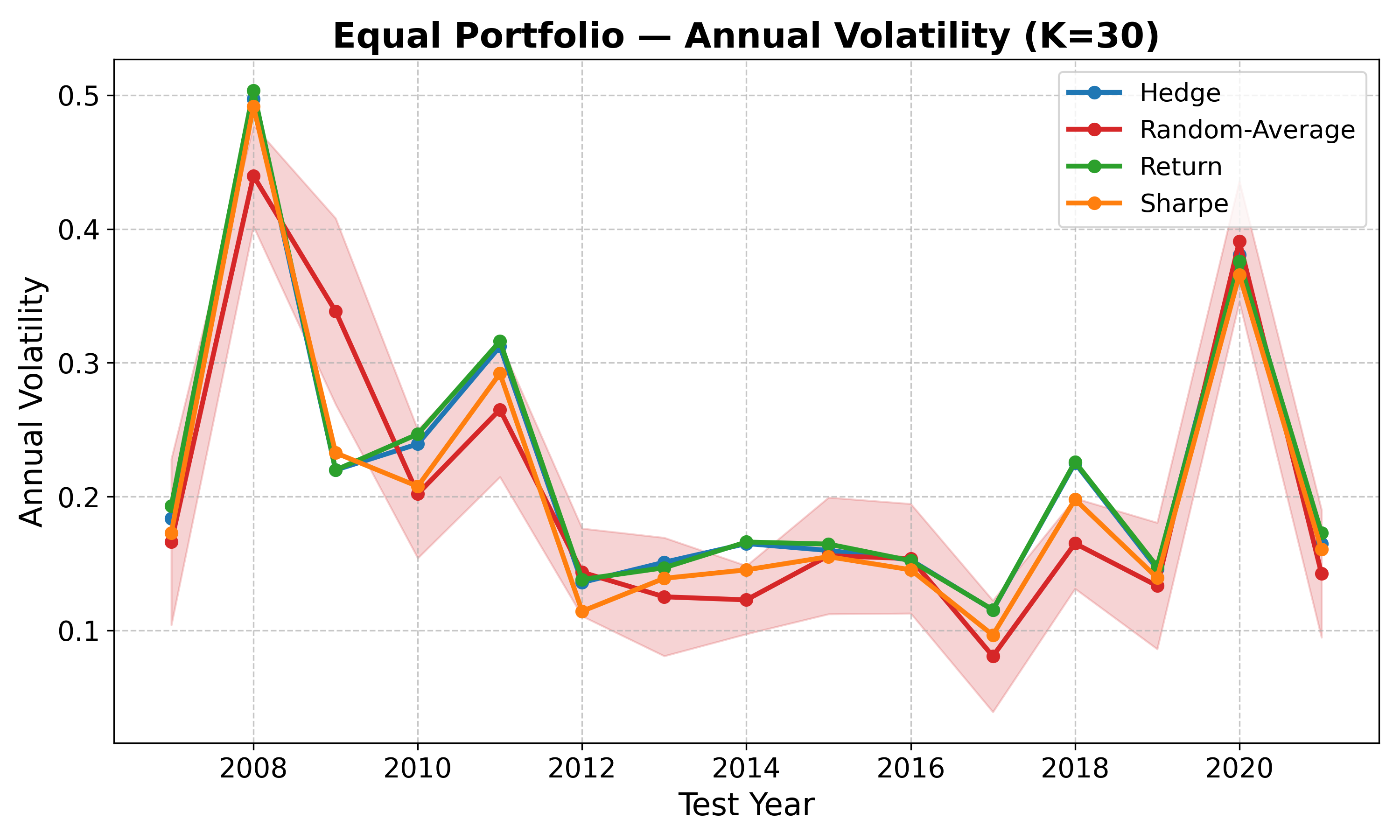}} 
    ~~~
    \subfigure[]{\includegraphics[width=0.40\textwidth]{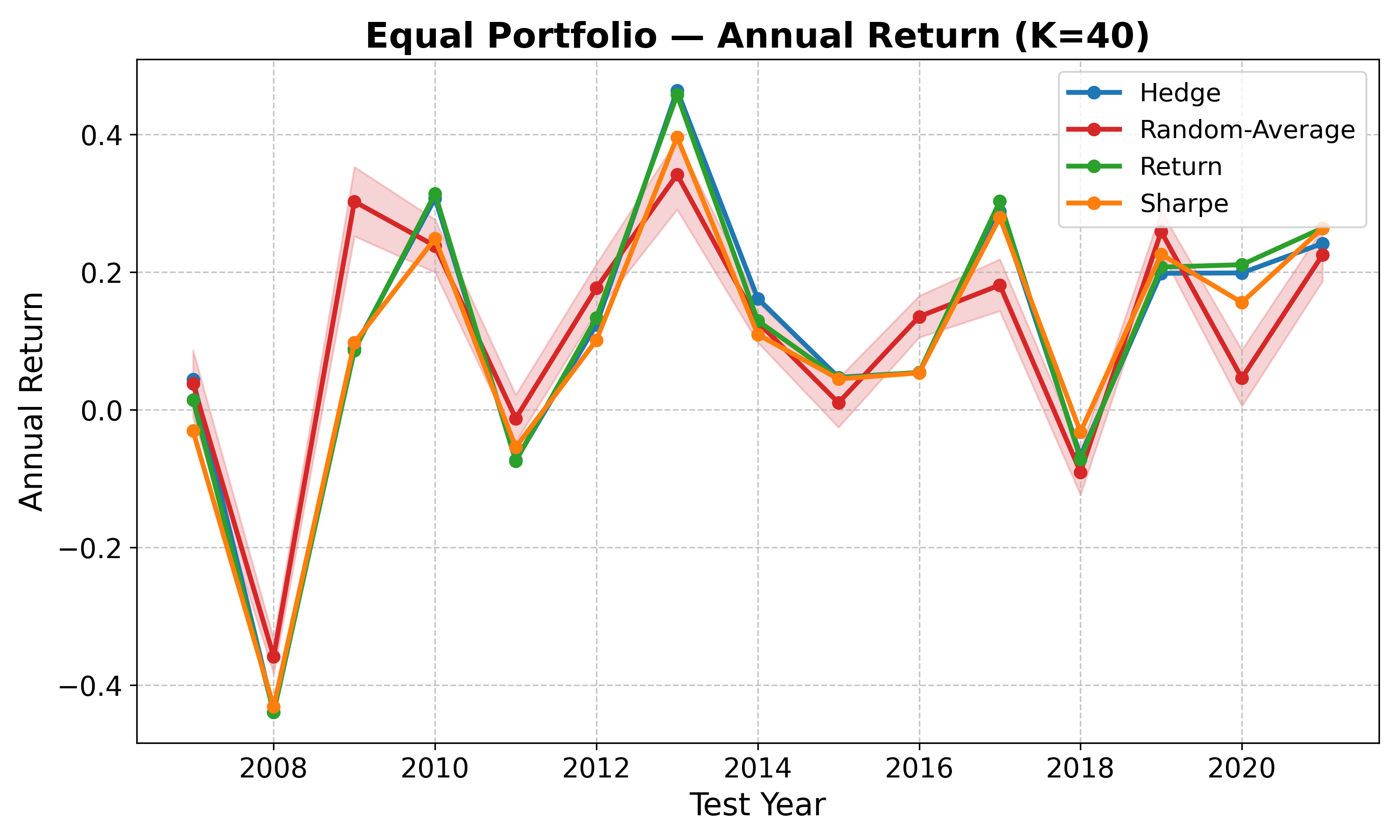}}
    ~~~
    \subfigure[]{\includegraphics[width=0.40\textwidth]{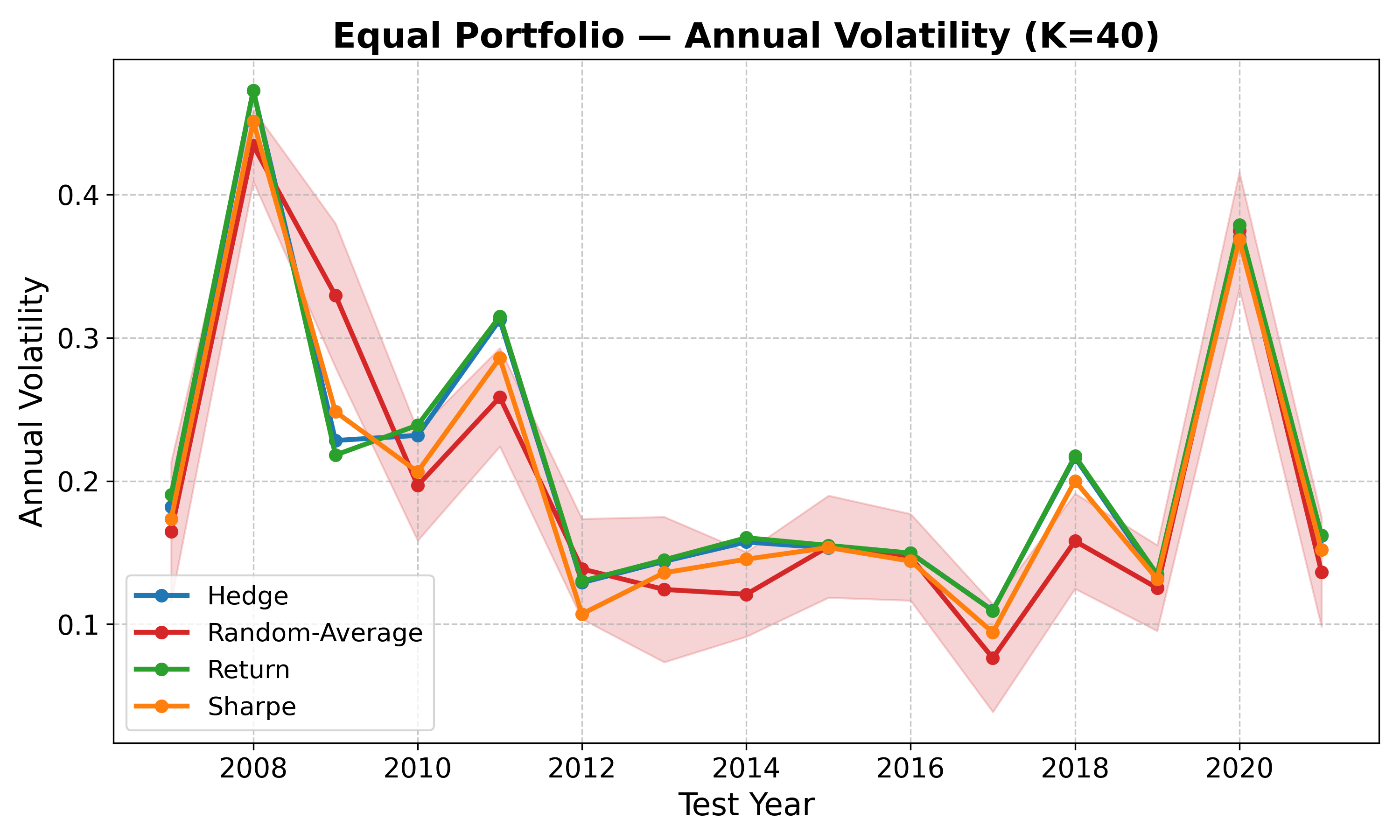}}
    ~~~
    \subfigure[]{\includegraphics[width=0.40\textwidth]{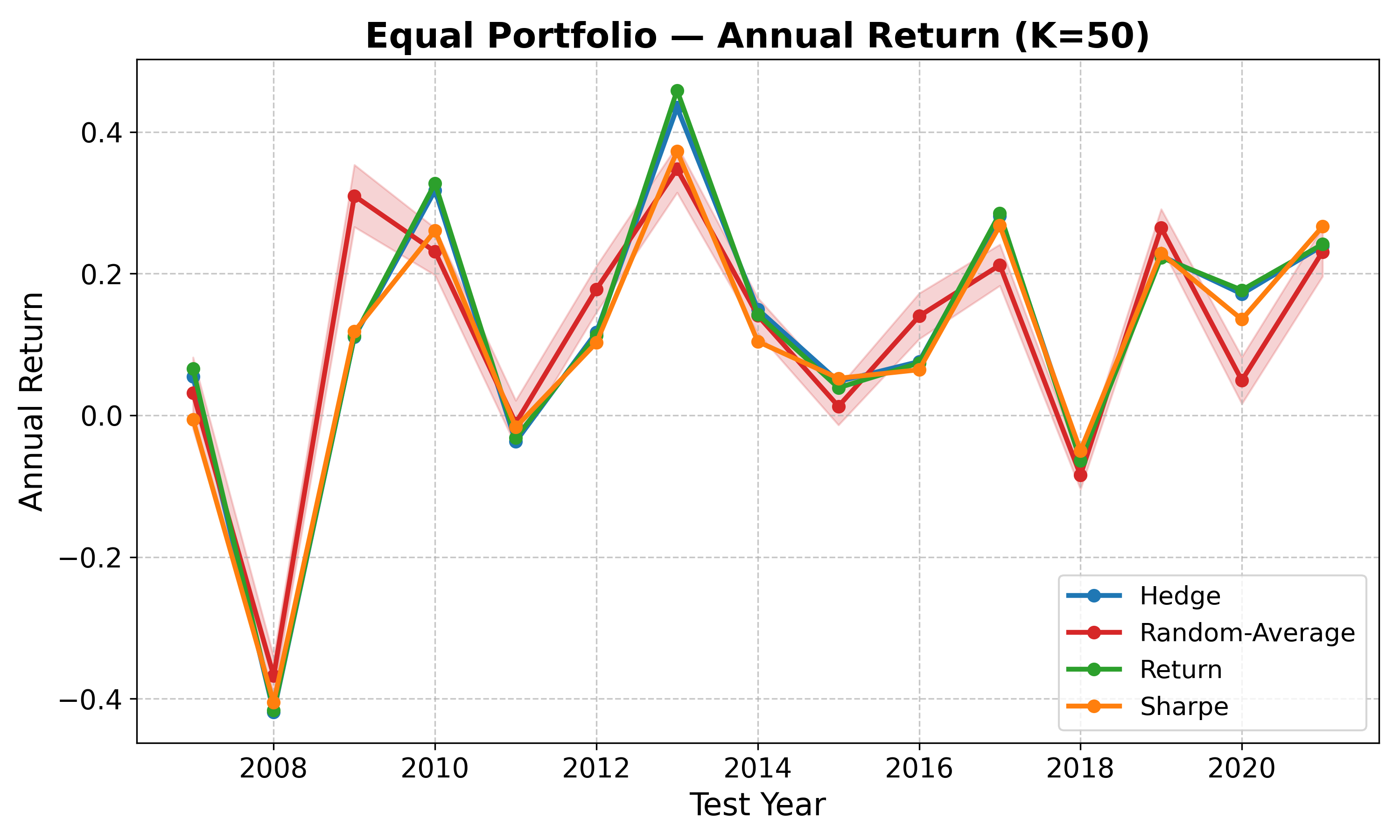}}
     ~~~
    \subfigure[]{\includegraphics[width=0.40\textwidth]{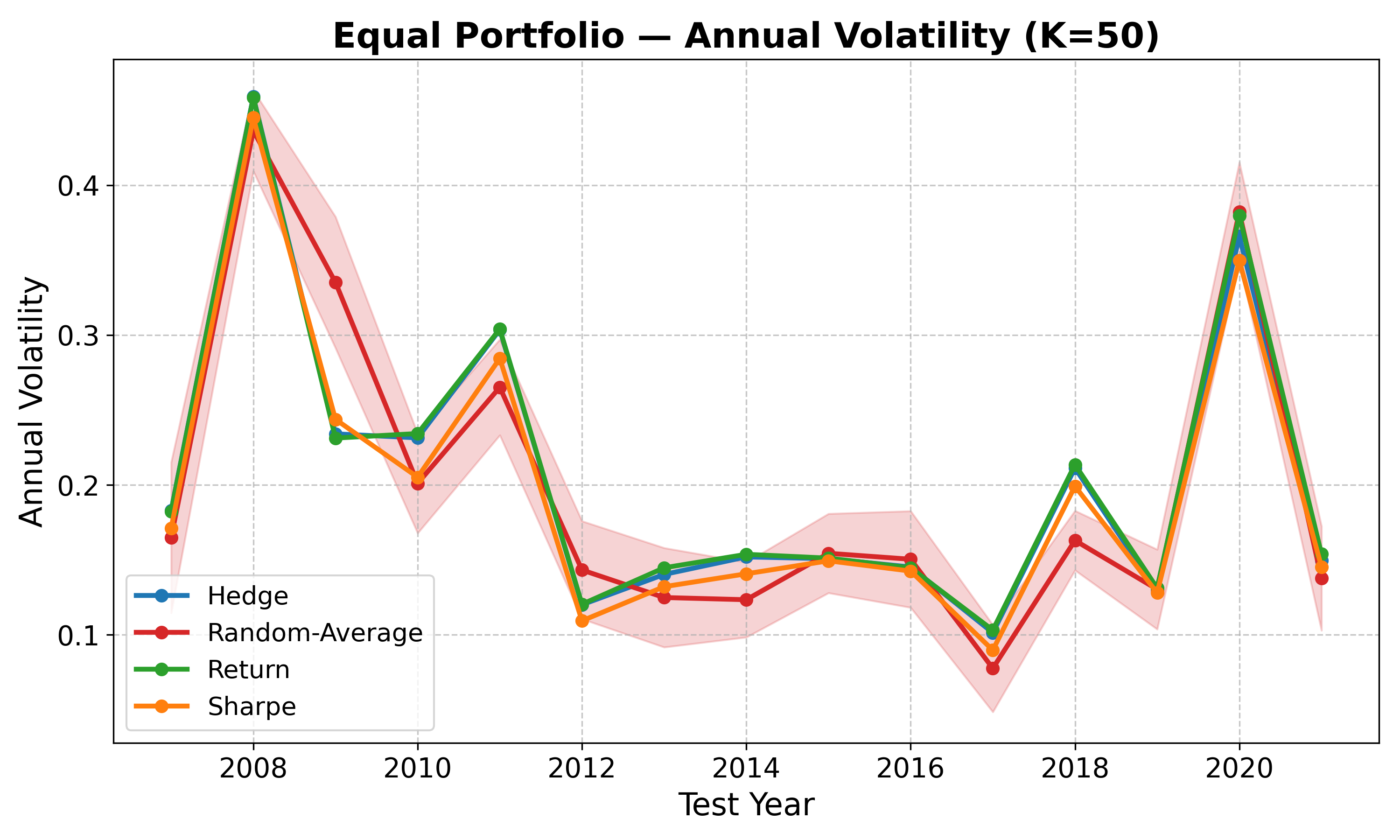}}
    \caption{Annual return (a), (c), (e), (g) and annual volatility (b), (d), (f), (g) comparison for equally weighted portfolio formation, setting the value of $K=20,30,40,$ and $50$, respectively.}
    \label{fig:equalreturn_compare}
\end{figure}

\begin{figure}[htp]
    \centering
    \subfigure[]{\includegraphics[width=0.40\textwidth]{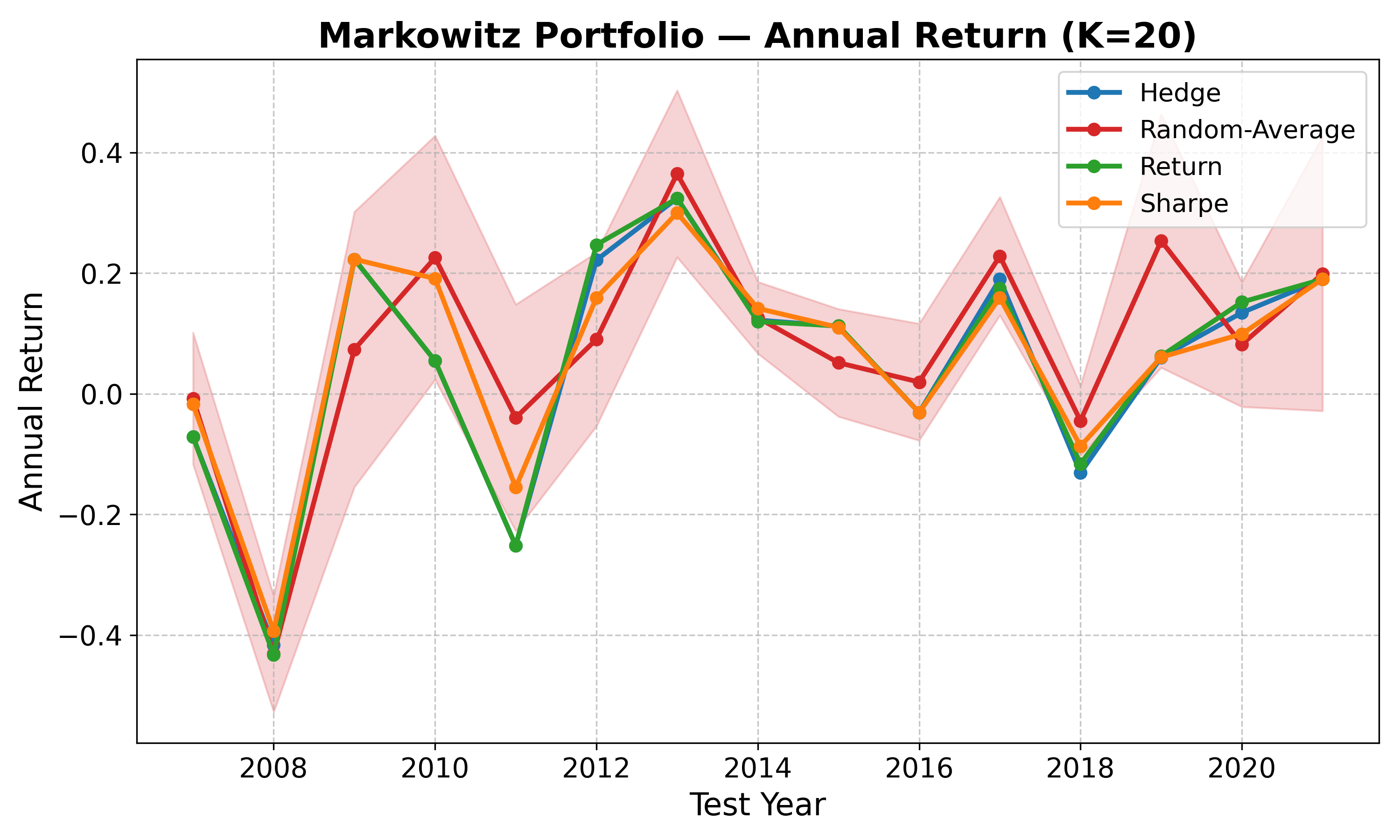}}
    ~ ~ ~
     \subfigure[]{\includegraphics[width=0.40\textwidth]{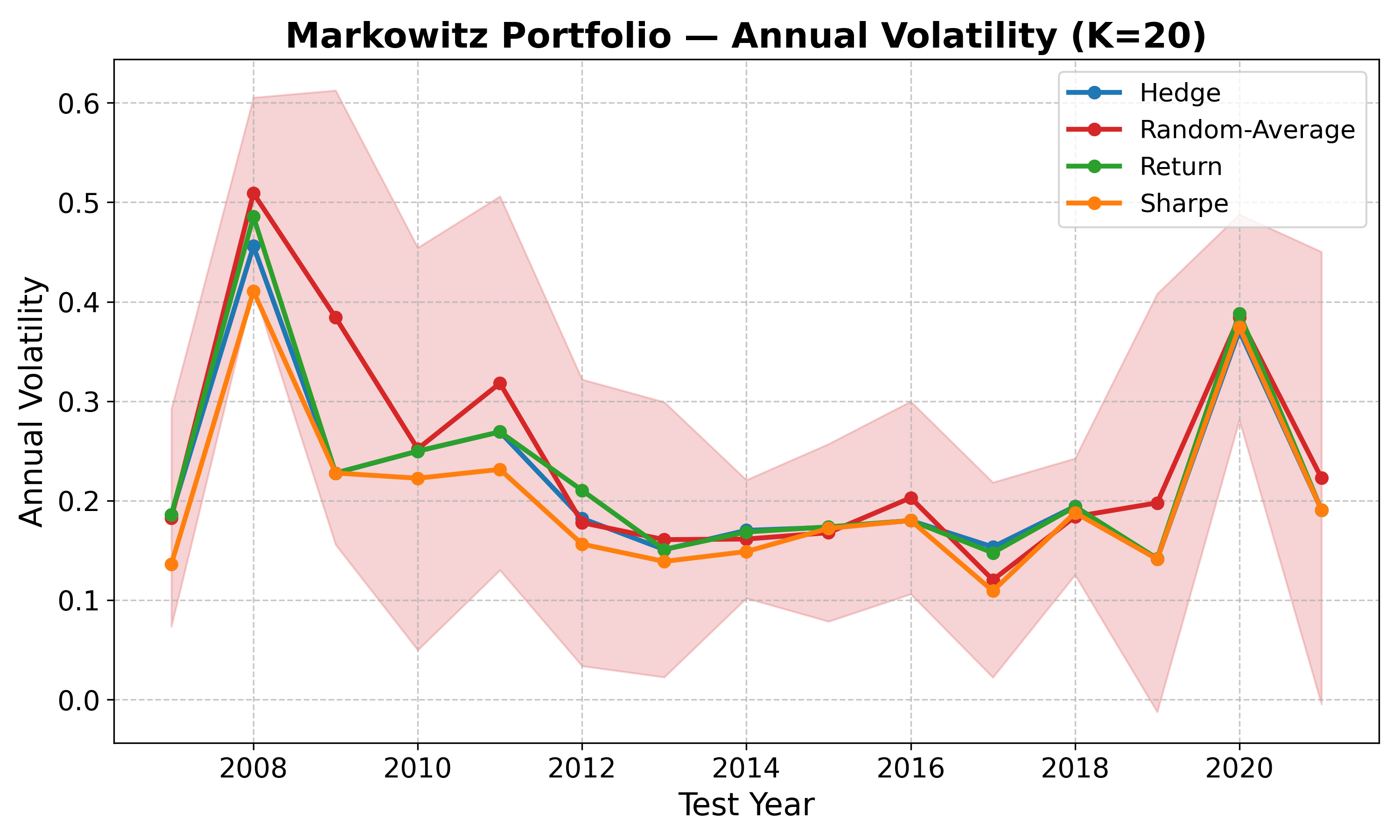}}
    ~ ~ ~
    \subfigure[]{\includegraphics[width=0.40\textwidth]{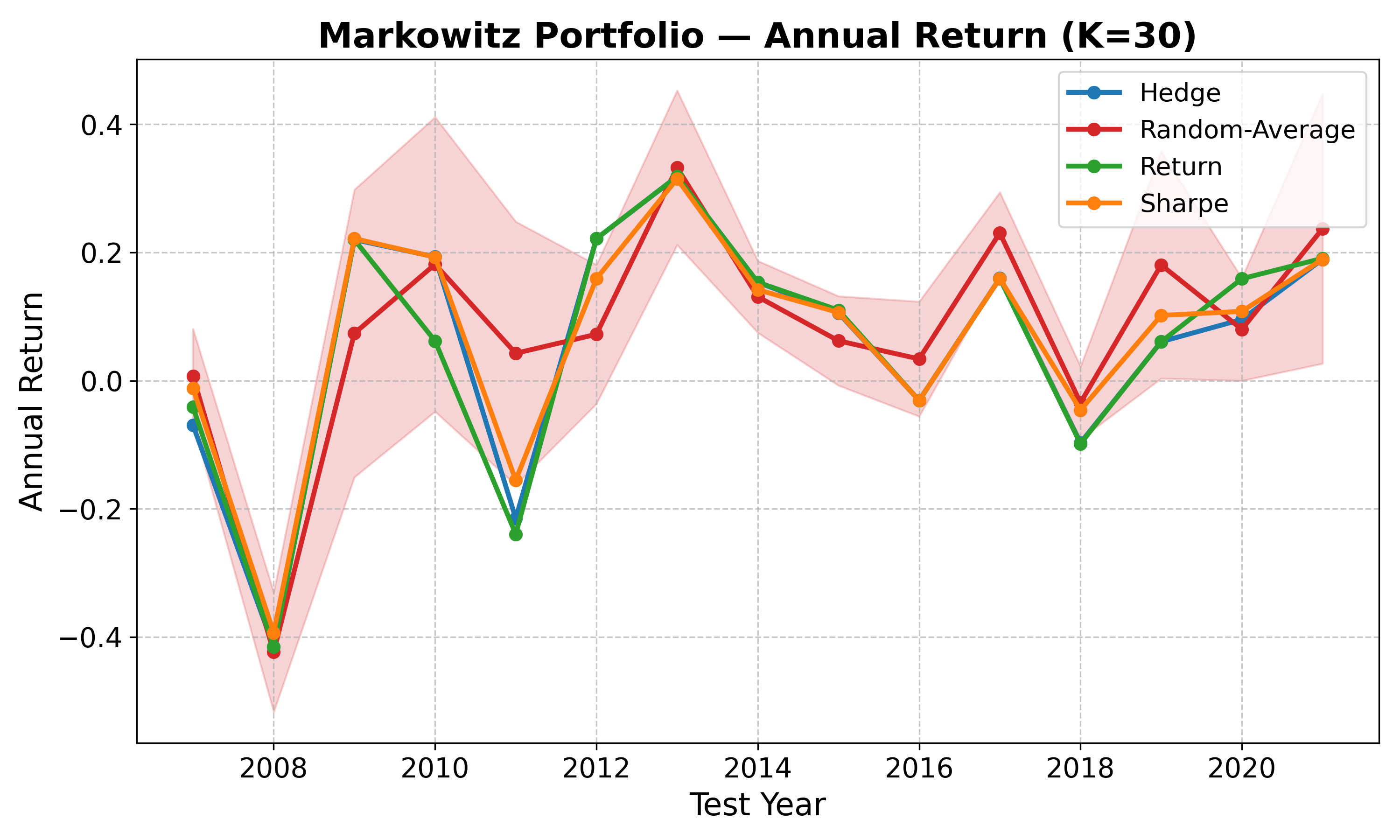}} 
    ~~~
     \subfigure[]{\includegraphics[width=0.40\textwidth]{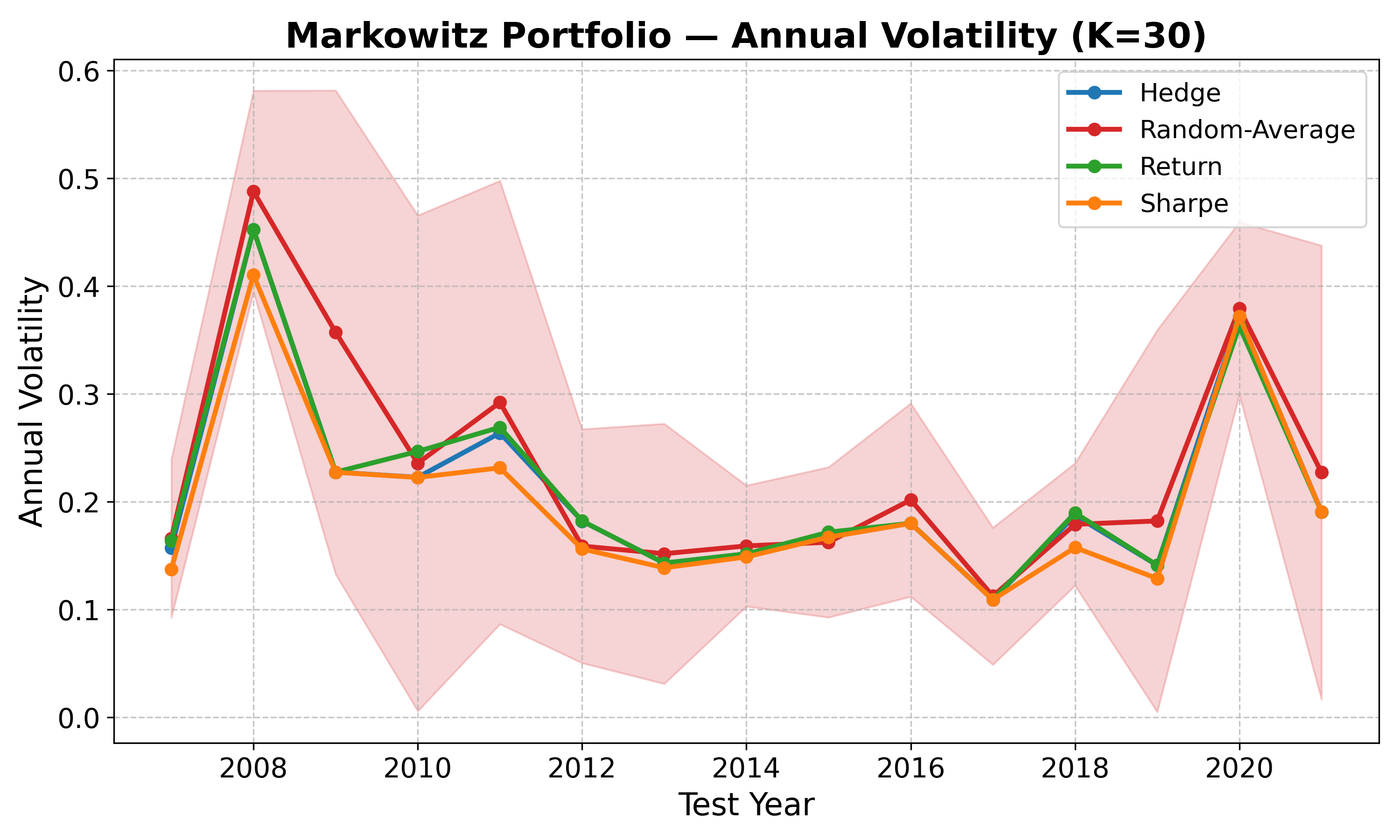}} 
    ~~~
    \subfigure[]{\includegraphics[width=0.40\textwidth]{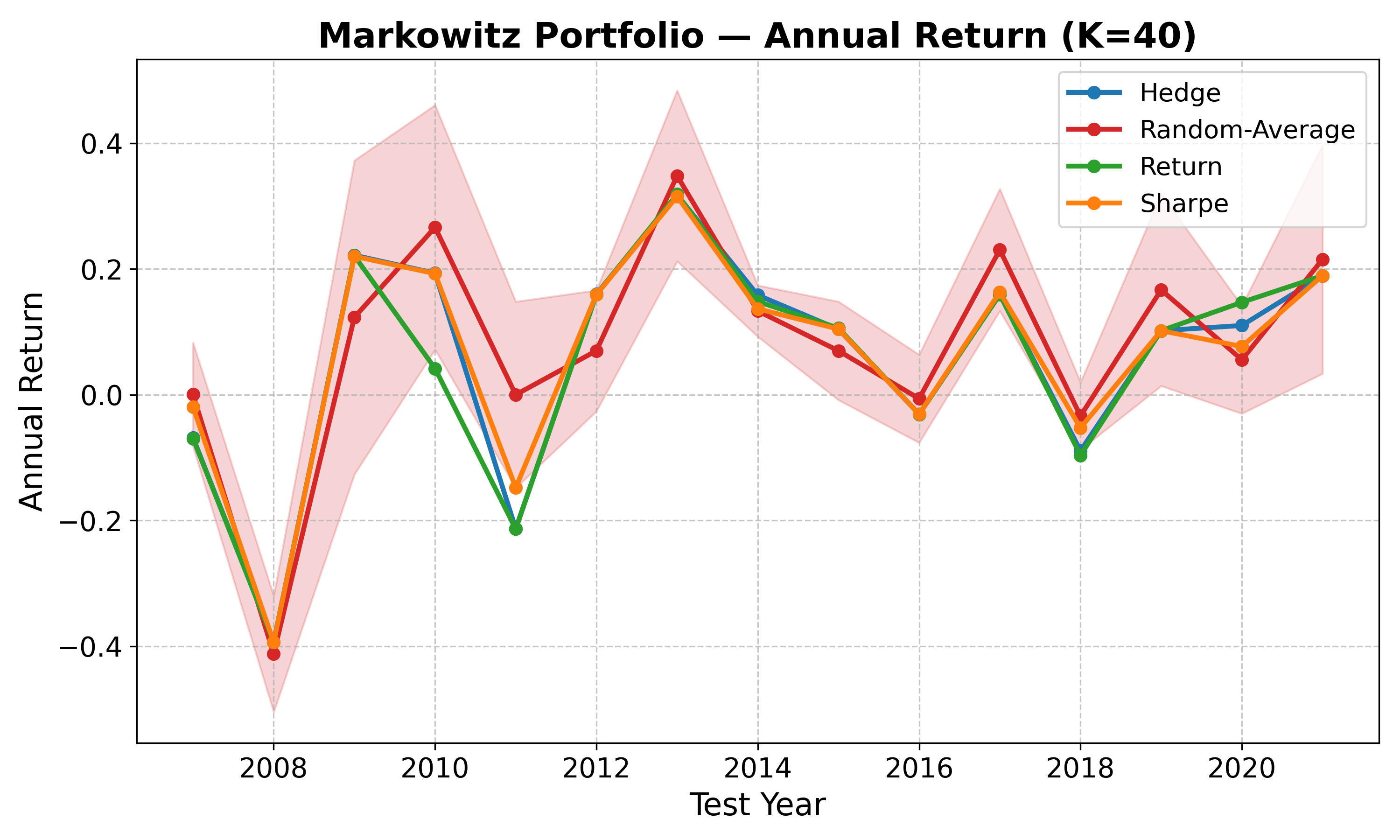}}
     ~~~
    \subfigure[]{\includegraphics[width=0.40\textwidth]{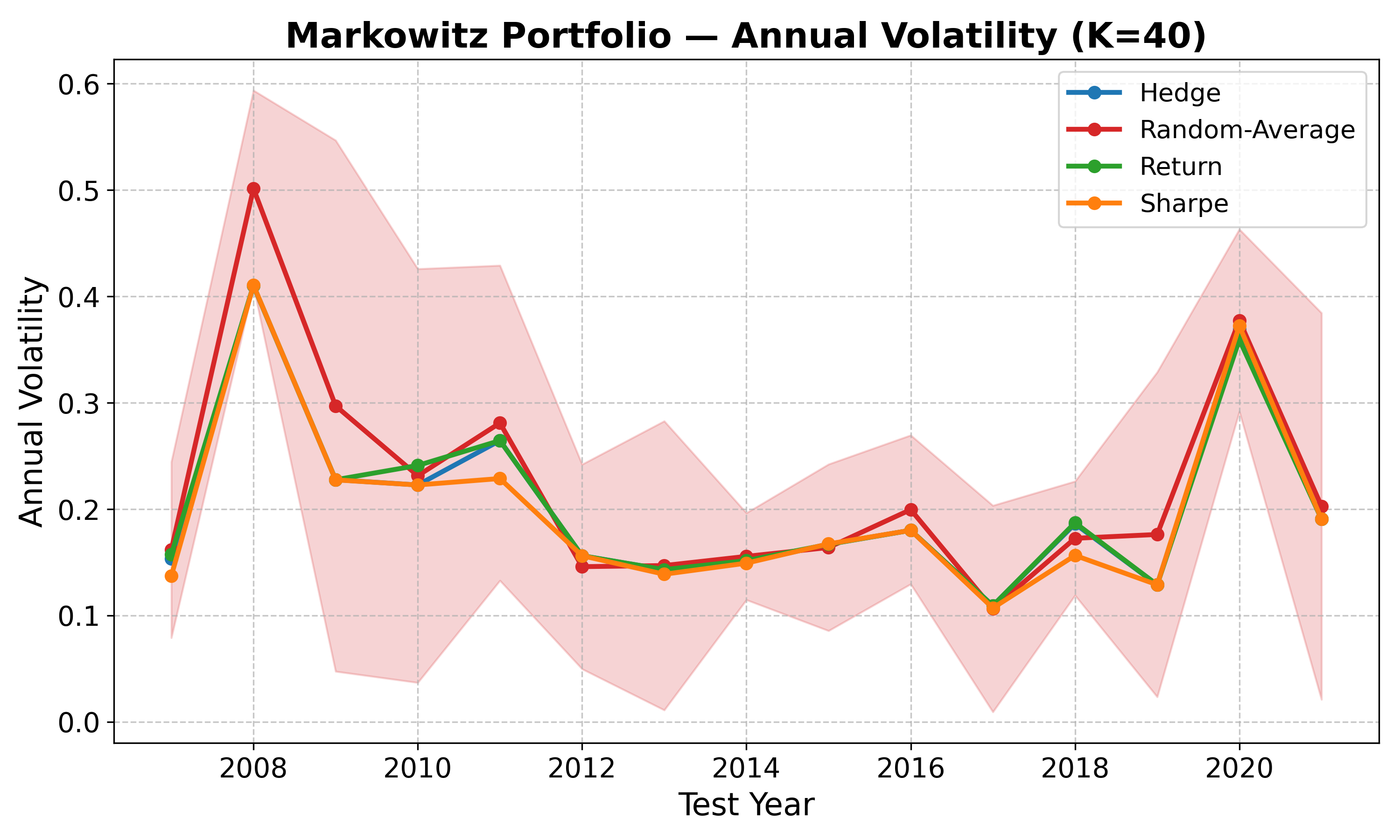}}
    ~~~
    \subfigure[]{\includegraphics[width=0.40\textwidth]{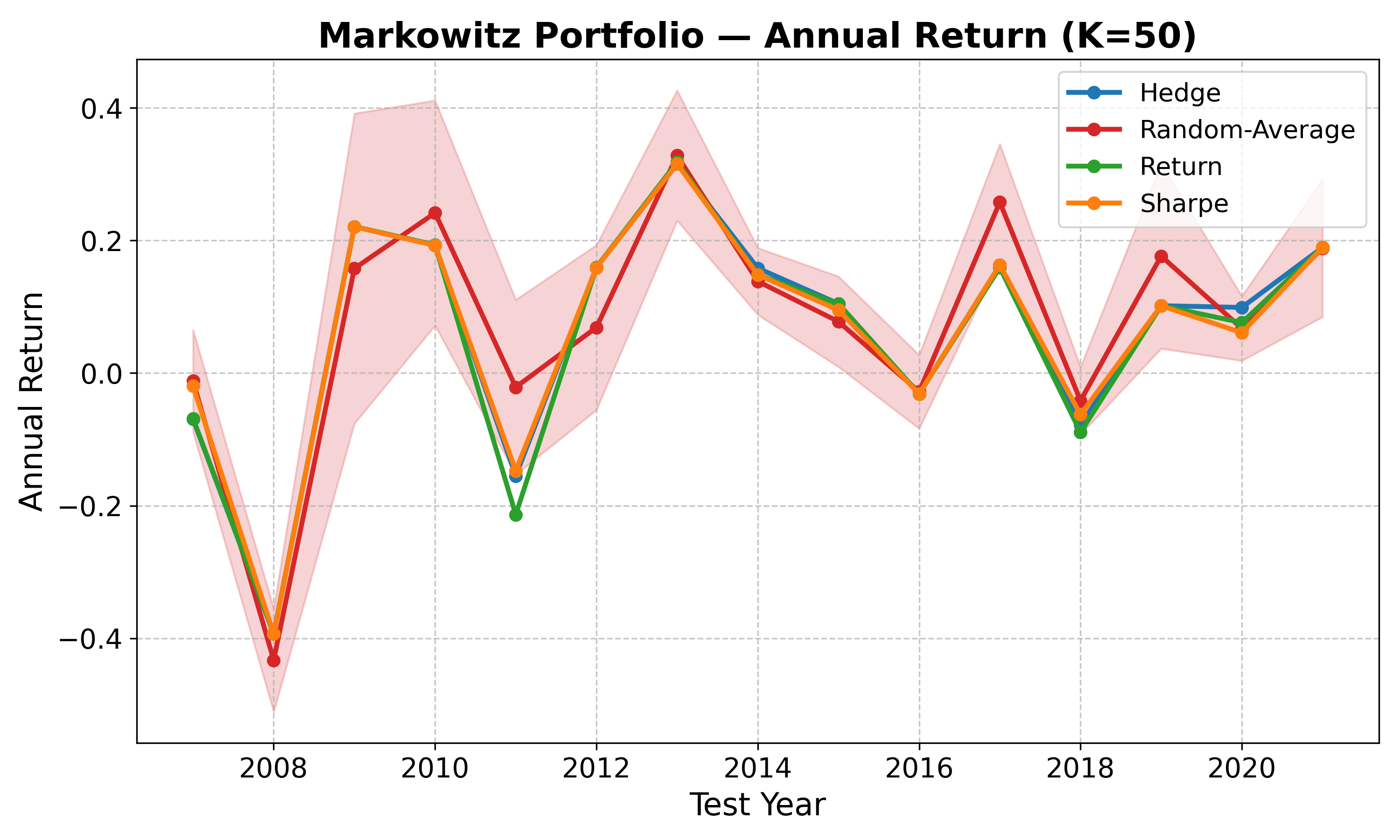}}
     ~~~
    \subfigure[]{\includegraphics[width=0.40\textwidth]{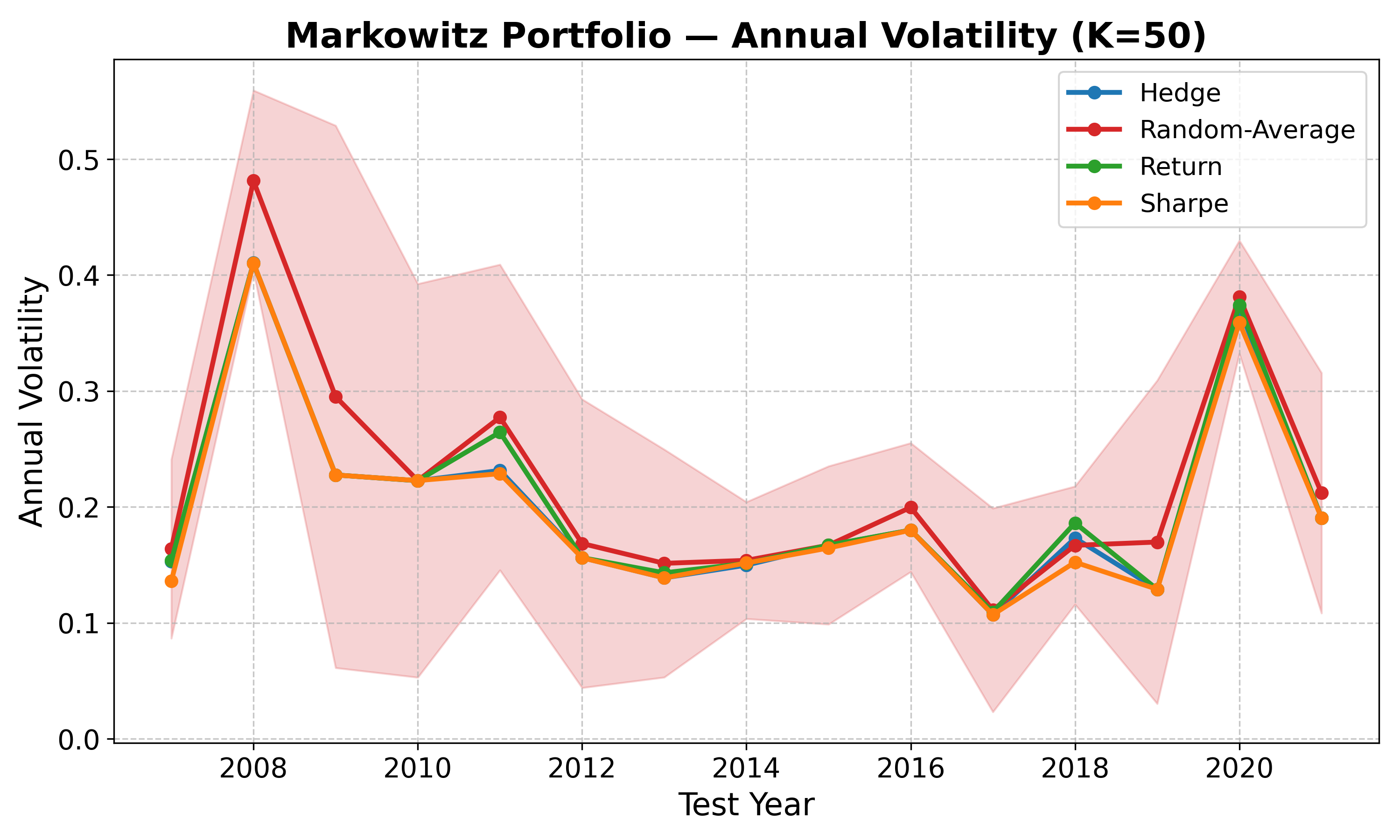}}
    \caption{Annual return (a), (c), (e), (g) and annual volatility (b), (d), (f), (g) comparison for Markowitz-type portfolio obtained by maximizing the Sharpe ratio, setting the value of $K=20,30,40,$ and $50$, respectively. }
    \label{fig:Markreturn_compare}
\end{figure}

\subsection{Comparison of out-of-sample evaluation}

We evaluate the proposed hedge-score-based framework for constructing portfolios of $K$ assets and compare it with alternative selection methods that capture different aspects of asset behavior. To mimic realistic investment conditions, we adopt a rolling out-of-sample protocol as in the previous section: for each year $y \in \{2006,\ldots,2020\}$, asset selection is performed using data from year $y$, and the resulting portfolio is evaluated on the subsequent year $y+1$. This generates a sequence of out-of-sample performance measures while avoiding look-ahead bias.The proposed hedge-score-based method ranks assets according to the hedge-adjusted score
$s_i = h_i \mu_i$. For comparison, we also consider return-based selection, which ranks assets by mean return, and Sharpe-based selection, which ranks assets by individual risk-adjusted performance. As an additional benchmark, we include a random selection strategy in which subsets of size $K$ are sampled uniformly without replacement, with performance averaged over multiple independent trials to ensure robustness.

Given a selected subset, portfolios are constructed using two standard allocation schemes. The equally weighted portfolio assigns uniform weights $w_i = 1/K$, providing a simple and stable benchmark. In addition, we consider a Markowitz-type portfolio obtained by maximizing the Sharpe ratio subject to long-only constraints $w_i \geq 0$ and $\sum_i w_i = 1$, thereby incorporating risk-return trade-offs within the reduced asset universe. Portfolio performance is evaluated using daily log-returns in the test period. To assess robustness with respect to the portfolio size, the analysis is repeated across multiple values of $K \in \{20,30,40,50\}$, enabling a systematic study of the trade-off between diversification and concentration.

Empirically, we observe that the hedge-score-based and mean return-based selections of $K$ assets provides competitive and often superior performance in terms of annual returns across most of the years; see Figures~\ref{fig:equalreturn_compare} (a), (c), (e), and (g) for the equally weighted portfolio. This phenomenon can be justified from the fact that the selection of a set of assets is decided by $\sum_{i} h_i\mu_i$ and $\sum_i \mu_i$ for the hedge score based and the only return-based strategies, respectively. While the average annual return over $30$ random selection of assets occasionally outperforms in specific years (e.g., 2008, 2009, and 2019), this comes at the cost of significantly higher variance of both annual return and annual volatility, as reflected by the wide shaded regions in the corresponding plots.  The annual volatility remains relatively high for both hedge-score-based and return-based selections across all values of $K$; see Figures~\ref{fig:equalreturn_compare} (b), (d), (f), and (h). 

For Markowitz-type portfolios based on Sharpe ratio maximization, the average annual return for $30$ random selection of assets continues to display high annual returns and volatilities, accompanied by large dispersion, as illustrated in Figures~\ref{fig:Markreturn_compare} (a)--(h). Notably, the variance associated with random selection is significantly larger in the Markowitz setting compared to the equally weighted case. This indicates that random selection does not yield a stable or reliable strategy for dimensionality reduction, particularly when combined with optimization-based portfolio construction.

Finally, the numerical simulations indicate that equally weighted portfolio formation based on hedge-score-driven dimensionality reduction provides a consistent edge over alternative selection methods. This observation motivates the exploration of motif-based, higher-moment portfolio construction as formulated in Eqn.~(\ref{eqn:opts2}), which remains a challenging direction for future research. Combined with a rigorous rolling out-of-sample evaluation and systematic comparison against standard baselines, the proposed framework offers a principled approach to incorporating network-based structural information into portfolio optimization.\\

\noin{\bf Conclusion.} 
We develop a unified framework linking signed network representations with higher-order combinatorial structure and portfolio optimization. We propose a framework for dimensionality reduction in portfolio optimization based on a complete signed network representation of a financial market constructed per trading-day basis. Unlike conventional approaches, the framework does not explicitly rely on covariance or correlation matrices. Instead, the sign of an edge between a pair of assets is determined by whether their daily log returns move in the same direction relative to their respective mean return values over a given time window. By considering a weighted signed graph representation of the market, we demonstrate that negative edges naturally capture hedging relationships in the sense of Markowitz’s portfolio theory. Building on this insight, we introduce a hedge score for each asset over a specified time window, which is then used to formulate a dimensionality reduction strategy based on both hedge scores and mean returns. A primary motivation of dimensionality reduction is equally weighted portfolio formation, which a competitive alternative to Markowitz's portfolio formation.

To further incorporate higher-order moments, we develop a combinatorial framework that links skewness and kurtosis to structural patterns in signed graphs. This leads to dimensionality reduction criteria based on identifying subsets of assets that induce complete subgraphs with a high density of balanced triangles and balanced $4$-cliques, denoted as $K_4^{B_2}$ (see Figure \ref{fig:bK4}), of a specific signed configuration. In this context, we show that finding $K_4^{B_2}$-dense subgraph of size $K$ is an NP-hard problem for complete signed networks. Finally, we evaluate the performance of the proposed dimensionality reduction framework through empirical analysis, comparing Markowitz mean-variance optimization and equally weighted portfolio construction within the reduced asset universe defined by hedge scores. The evaluation is conducted on a dataset of 199 assets from the S\&P 500 market, aligned with Google stock data, spanning the period from 2006 to 2021.

Our study of combinatorial criteria for higher-order moments within a signed graph representation of financial markets introduces NP-hard problems in signed graph theory. We establish that developing efficient approximation methods for motif-dense subgraph discovery in signed networks can play a significant role in advancing computational finance, particularly for dimensionality reduction of portfolio optimization. \\

\noin{\bf Acknowledgment.}
The author thanks Hannes Leipold, Sarvagya Upadhyay,  Hirotaka Oshima, and Yasuhiro Endo for their insightful comments and discussion. The author thanks Steven Kordonowy for thoughtful comments. The authors thank the anonymous reviewers for their constructive feedback and valuable suggestions, which have significantly improved the manuscript.

\bibliographystyle{unsrt}
\bibliography{reff}
\end{document}